\documentclass[12pt,a4paper]{amsart}
\usepackage{amsthm}
\usepackage{amsmath,amscd,upref}

\newtheorem{Theorem}{Theorem}[section]
\newtheorem{Proposition}[Theorem]{Proposition}
\newtheorem{Lemma}[Theorem]{Lemma}
\newtheorem{Corollary}[Theorem]{Corollary}

\theoremstyle{definition}

 \theoremstyle{remark}

\makeatletter
       \def\@makefnmark{%
               \leavevmode
               \raise.9ex\hbox{\check@mathfonts
                       \fontsize\sf@size\z@\normalfont%
                               \@thefnmark}%
       }
       \makeatother

 \makeatletter
    
    \@addtoreset{equation}{section}
  \makeatother
\newfont{\BBb}{msbm10 scaled\magstep1}
\newfont{\sBbb}{msbm7 scaled\magstep1}

\newcommand{\ZZ}{\mbox{\BBb Z}}
\newcommand{\sZZ}{\mbox{\sBbb Z}}

\newcommand{\NN}{\mbox{\BBb N}}

\def\Sum{\textstyle\sum\limits}

\def\Frac#1#2{\displaystyle\frac{\mathstrut\raisebox{-.3ex}{$\;#1\;$}}{\;#2\;}}
\def\Sum{\displaystyle\sum}
\def\wtilde{\widetilde}

\def\Cotor{{\rm Cotor}}

\def\KER{\mbox{}\hskip11pt{\rm Ker}\;}
\def\Coker{{\rm Coker}\;}

\def\mathvskip{\vskip\abovedisplayskip}
\def\hpm{\hphantom{-}}

\begin{document}

\title[Cohomology mod 3 of the classifying space of $E_6$]{Cohomology mod 3 of the classifying space of the exceptional Lie group $E_6$, 
I : structure of Cotor
}
\thanks{Partially supported by the Grant-in-Aid for Scientific Research (C) 
21540104, Japan Society for the Promotion of Science.
}
\author[M.Mimura]{Mamoru MIMURA
}
\address{Department of Mathematics
Faculty of Science Okayama University;
3-1-1 Tsushima-Naka, Okayama
700-8530, Japan}
\email{mimura@math.okayama-u.ac.jp}

\author[Y.Sambe]{Yuriko SAMBE}
\address{
School of Business Administration
Senshu University;
2-1-1 Tama-ku, Kawasaki
214-8580, Japan}
\email{thm0166@isc.senshu-u.ac.jp}

\author[M.Tezuka]{Michishige TEZUKA}
\address{Department of Mathematical Sciences
Faculty of Science University of Ryukyus; 
Nishihara-cho, Okinawa 903-0213, Japan}
\email{tez@math.u-ryukyu.ac.jp}


\maketitle

\begin{abstract}

We study the structure of the $E_2$-term of the Rothenberg-Steenrod spectral sequence converging to the classifying space of the compact, connected, simply connected, exceptional Lie group of rank $6$.
\vspace{4pt}
\vspace{4pt}\\
Key words and phrases, exceptional Lie group, spectral sequence, Cotor
\end{abstract}

\section{Introduction}
Let $E_6$ be the compact, connected, simply connected, exceptional 
Lie group of rank 6 and $BE_6$ its classifying space. 
Then we have the Rothenberg-Steenrod spectral 
sequence $\{E_r,d_r\}$ (\cite{RS}) such that
\[
  E_2 = \textstyle \Cotor_{H^*(E_6; \sZZ_3)}^{}(\ZZ_3,\ZZ_3)\quad
    \mbox{and}\quad
    E_\infty = \mbox{gr}\, H^*(BE_6;\ZZ_3).\]

The $E_2$-term was calculated by Mimura and Sambe \cite{MS}, where some 
details of the calculations are somewhat omitted. 
This is one of the reasons why we give here a detailed account of them.
After the calculations of the $E_2$-term, the collapsing 
of the spectral sequence was proved by Kono and Mimura \cite{KM}; 
thus one can obtain the module structure of $H^*(BE_6;\ZZ_3)$.
(See also Toda \cite{T}, in which the collapsing is shown by a different method.)
Further in order to determine the ring structure of $H^*(BE_6;\ZZ_3)$, 
we will need more information (see \cite{MSTo}), that is, 
we need to add supplement to the results of \cite{MS}. 
This is the main reason for the present work.

The paper is organized as follows.
In Section 2 we state our results.
In Section 3 we prove Theorem 2.1 and in Section 4 we prove Theorems 2.3.A and 2.3.B.
In Section 5, we compare our calculation to the May spectral sequence.

Thus the present work is essentially a supplement, 
and hence a more detailed version, of \cite{KM} and \cite{MS}. 
The results of the present work are announced in \cite{MSTe}. 

In the series of our work, we will determine  $H^*(BE_6;\ZZ_3)$. Continuation follows.

Throughout the paper, the coefficient field is always $\ZZ_3$, 
which is hereafter omitted.

\section{Results}
Let us start with

\bigskip
{\bf Notation} \quad 
For a commutative ring $R$, we denote by
$R\{x_1, \ldots, x_n\}$
a free $R$-module generated by elements $x_1$, \ldots, $x_n$. 
For a graded algebra $S$, we denote by $S^ + $ the ideal generated by the elements with positive degrees.
We also denote by $T(x_1, \ldots, x_n)$ a tensor algebra generated 
by elements $x_1$, \ldots, $x_n$, and by  
$\Delta(x)$ an algebra additively isomorphic 
to $\ZZ_3 \oplus \ZZ_3\{x\}$. 

The following is essentially the restatement of Theorem 5.20 of \cite{MS}.

\begin{Theorem}\label{21}
$\Cotor_{H^*(E_6)}^{}(\ZZ_3,\ZZ_3)$ {\it is 
additively isomorphic to the following module}\/:
$$(C \oplus D) \otimes \ZZ_3[x_{36},x_{48},x_{54}],$$
$\begin{array}{ll}
{\it where} 
\  &C = \ZZ_3[a_4,a_8,a_{10}]\{1,y_{20},y_{20}^2,y_{22},
                        y_{22}^2,y_{20}y_{22},y_{58},y_{60},y_{76}\} 
                                                \\ & \mbox{}\hskip12em
           \oplus \ZZ_3[a_8,a_{10}]\{y_{26},y_{26}^2,
                                y_{20}y_{26},y_{22}y_{26},y_{64}\}\\
{\it and} 
    & D = \ZZ_3[x_{26}]^+ \{1, a_4, a_8,a_{10},y_{20},y_{22},
                        a_{10}y_{20},y_{26}\} \\ & \mbox{}\hskip12em
           \oplus \ZZ_3[x_{26}]\{a_9,y_{21},y_{25},y_{27},
                     y_{21}a_8,y_{21}a_{10},y_{25}a_{10},y_{21}y_{26}\}.
\end{array}$
\end{Theorem}

By an easy calculation one obtains

\begin{Corollary}\label{22}
{\it The Poincar\'e polynomials PS of them are given respectively by }
$$
\begin{array}{ll}
PS(C \otimes \ZZ_3[x_{36}, x_{48}, x_{54}]) 
        & = \Frac{g(t)}{(1 - t^4)(1 - t^8)(1 - t^{10})}
            \cdot \Frac{1}{(1 - t^{36})(1 - t^{48})(1 - t^{54})},\\[5mm]
PS(D \otimes \ZZ_3[x_{36}, x_{48}, x_{54}]) 
        & = \Frac{h(t)}{1 - t^{26}}
            \cdot \Frac{1}{(1 - t^{36})(1 - t^{48})(1 - t^{54})},
\end{array}$$
$\begin{array}{ll}
{\it where} \ & g(t) = 1 + t^{20} + t^{22} + t^{26} - t^{30} + t^{40} 
        + t^{42} + t^{44} + t^{46} + t^{48} - t^{50} - t^{56} + t^{58} 
\hspace{16mm} \\ & \mbox{}\hfill
                  + t^{60} + t^{64} - t^{68} + t^{76}\\
{\it and}    & h(t) = t^9 + t^{21} + t^{25}  + t^{26} + t^{27} + t^{29} 
                  + t^{30} + t^{31} + t^{34} + t^{35} + t^{36} + t^{46}
                  + t^{47}
\\ & \mbox{}\hfill
                  + t^{48} + t^{52} + t^{56}. 
\end{array}$
\end{Corollary}

The following two theorems will 
play an important role when determining  
the ring structure of $H^*(BE_6)$ (see \cite{MSTo}).

\bigskip
{\bf Theorem 2.3.A} \quad 
{\it The extension}
$$ 0 \to D \otimes \ZZ_3[x_{36},x_{48},x_{54}]
     \to \Cotor_{H^*(E_6)}^{} (\ZZ_3, \ZZ_3)
     \to  C \otimes \ZZ_3[x_{36},x_{48},x_{54}] \to 0$$
{\it is split as algebras}. 

\bigskip
{\bf Theorem 2.3.B}  \quad 
{\it We have an algebra isomorphism }

\mathvskip
                                                \mbox{}\hskip4em 
$\Cotor_{H^*(E_6)}^{}(\ZZ_3,\ZZ_3) 
        \cong \ZZ_3[a_9,x_{26},y_{21},y_{25},y_{27},
                        a_4,a_8,a_{10},x_{36},
                                                \\ \mbox{}\hskip19.2em
                 x_{48},x_{54},y_{20},y_{22},y_{26},y_{58},y_{60},y_{64},y_{76}]/I$, \\
{\it where $I$ is the ideal generated by the following}

$\begin{array}{lll}
                \hbox to 10mm{} & & \\[-6mm]
\rm{\hphantom{ii}i)}
& a_4y_{26}      & = \mbox{} -a_8y_{22} +a_{10}y_{20}, \\
& a_4y_{64}      & = \mbox{} -a_8y_{60} -a_8^3y_{22}^2
                                        -a_{10}y_{58}, \\ 
& y_{22}^3 & = \mbox{} -a_4^3x_{54} -a_4^2a_8^2a_{10}^2y_{22}
        +a_4a_8a_{10}y_{22}^2 +a_4a_{10}^2y_{20}y_{22} +a_{10}^3x_{36}, \\
& y_{20}^2y_{22} & = \mbox{} -a_4y_{58} +a_8^2a_{10}x_{36} , \\
& y_{20}y_{22}^2 & = \mbox{} -a_4y_{60} +a_8a_{10}^2x_{36} , \\
& y_{58}y_{22} & = \mbox{} -a_4y_{76} -a_4a_8^2y_{60}
        +a_8^3a_{10}^2x_{36} +a_8a_{10}x_{36}y_{26}, \\
& y_{60}y_{22} & = a_4^2x_{54}y_{20} +a_4a_8^2a_{10}^2y_{20}y_{22}
        +a_4a_8a_{10}y_{60} +a_4a_{10}^2y_{58}
                \\ & &\qquad \mbox{}
        +a_8^2a_{10}^3x_{36} -a_{10}^2x_{36}y_{26} , \\
& y_{76}y_{22} & = \mbox{} -a_4x_{54}y_{20}^2 -a_4^2a_8^2x_{54}y_{20}
        +a_4^2a_{10}^2x_{48}y_{22} +a_4a_8^4a_{10}^2y_{20}y_{22}  
                \\ && \qquad \mbox{}
        +a_4a_8^2a_{10}^2y_{58} +a_4a_8a_{10}y_{76}
        -a_8^4a_{10}^3x_{36} +a_{10}x_{36}y_{26}^2, \\
& y_{22}y_{26}^2 & = \mbox{} -a_4a_8^2x_{54} -a_8^4a_{10}^2y_{22}
        +a_8^2a_{10}y_{22}y_{26} +a_{10}y_{64} , \\
& y_{20}y_{22}y_{26} & = a_8y_{60} -a_{10}y_{58}, \\
& y_{22}^2y_{26} & = a_4^2a_8x_{54} +a_4a_8^3a_{10}^2y_{22}
        -a_8^2a_{10}y_{22}^2 -a_8a_{10}^2y_{20}y_{22} -a_{10}y_{60}, \\
& y_{64}y_{22} & = a_4^2a_8^3x_{54} -a_4a_8x_{54}y_{20}
        +a_4a_8^5a_{10}^2y_{22} -a_8^4a_{10}y_{22}^2
        +a_8^3a_{10}^2y_{20}y_{22}
                \\ && \qquad \mbox{}
        -a_8a_{10}^2y_{58} +a_{10}y_{76}, \\ 
& y_{20}^3 & = a_4^3x_{48} -a_4^2a_8^4y_{20}
        -a_4a_8^2y_{20}^2 +a_8^3x_{36}, \\
& y_{20}^2y_{22} & = \mbox{} -a_4y_{58} +a_8^2a_{10}x_{36}, \\
& y_{58}y_{20} & = \mbox{} -a_4^2x_{48}y_{22} -a_4a_8^2y_{58}
        +a_4a_8^4y_{20}y_{22} +a_8^2x_{36}y_{26} +a_8^4a_{10}x_{36}, \\
& y_{60}y_{20} & = \mbox{} -a_4y_{76} -a_4a_8^2y_{60} +a_8^3a_{10}^2x_{36}
        -a_8a_{10}x_{36}y_{26}, \\
& y_{76}y_{20} & = a_4x_{48}y_{22}^2 +a_4a_8^4y_{60}
        +a_8x_{36}y_{26}^2 -a_8^5a_{10}^2x_{36}, \\
                \hbox to 10mm{} & & \\[-6mm]
& y_{20}y_{26}^2 & = a_4a_{10}^2x_{48} +a_8y_{64}
        -a_8^3y_{22}y_{26} -a_8^4a_{10}^2y_{20}
        -a_8^2a_{10}y_{20}y_{26}, \\
& y_{20}^2y_{26} & = a_4^2a_{10}x_{48} -a_4a_8^4a_{10}y_{20}
        +a_8y_{58} -a_8^2a_{10}y_{20}^2 , \\
& y_{64}y_{20} & = a_4a_{10}x_{48}y_{22} +a_8y_{76}
        -a_8^3y_{60} -a_8^4a_{10}y_{20}y_{22} +a_8^2a_{10}y_{58}, \\
& y_{58}^2 & = \mbox{} -a_4^2x_{48}y_{60} -a_4^2a_8^2x_{48}y_{22}^2
        -a_4^2a_8^6y_{60} +a_4a_8^7a_{10}^2x_{36}
        +a_4a_8a_{10}^2x_{36}x_{48} +a_8^2x_{36}y_{64} , \\
& y_{58}y_{60} & = \mbox{} -a_4^4x_{48}x_{54} +a_4^3a_8^4x_{54}y_{20}
        +a_4^3a_8^2a_{10}^2x_{48}y_{22} +a_4^2a_8^2x_{54}y_{20}^2
        -a_4^2a_8^6a_{10}^2y_{20}y_{22}
                \\ && \qquad \mbox{}
        -a_4^2a_8^4a_{10}^2y_{58}
        -a_4^2a_8^3a_{10}y_{76} +a_4^2a_8a_{10}x_{48}y_{22}^2
        +a_4^2a_{10}^2x_{48}y_{20}y_{22} -a_4a_8^3x_{36}x_{54}
                \\ && \qquad \mbox{}
        +a_4a_8^6a_{10}^3x_{36} +a_4a_{10}^3x_{36}x_{48}
        -a_8^5a_{10}^2x_{36}y_{22} -a_8^2a_{10}^2x_{36}y_{20}y_{26}
        -a_8a_{10}x_{36}y_{64} , \\
& y_{58}y_{76} & = a_4^3x_{48}x_{54}y_{20}
        +a_4^3a_8^4a_{10}^2x_{48}y_{22} -a_4^2a_8^4x_{54}y_{20}^2
        -a_4^2a_8^8a_{10}^2y_{20}y_{22} +a_4^2a_8^7a_{10}y_{60}
                \\ && \qquad \mbox{}
        -a_4^2a_8^6a_{10}^2y_{58} +a_4^2a_8^5a_{10}y_{76}
        +a_4^2a_8^2a_{10}^2x_{48}y_{20}y_{22} +a_4^2a_8a_{10}x_{48}y_{60}
                \\ && \qquad \mbox{}
        +a_4^2a_{10}^2x_{48}y_{58}
        +a_4a_8^2a_{10}^3x_{36}x_{48}
        +a_8^3x_{36}x_{54}y_{20} -a_8^6a_{10}^3x_{36}y_{20}
                \\ && \qquad \mbox{}
        +a_8^4a_{10}^2x_{36}y_{20}y_{26}
        -a_8a_{10}^2x_{36}x_{48}y_{22}
        -a_{10}^3x_{36}x_{48}y_{20} , \\
& y_{58}y_{26} & = \mbox{} -a_4a_{10}x_{48}y_{22} +a_8y_{76} +a_8^3y_{60}
        +a_8^4a_{10}y_{20}y_{22} -a_8^2a_{10}y_{58} , \\
& y_{58}y_{64} & = a_4^3a_8x_{48}x_{54} -a_4^2a_8^5x_{54}y_{20}
        +a_4a_8^3x_{54}y_{20}^2 -a_4a_8^6a_{10}y_{60} -a_4a_8^4a_{10}y_{76}
                \\ && \qquad \mbox{}
        -a_4a_8a_{10}^2x_{48}y_{20}y_{22} +a_4a_{10}x_{48}y_{60}
        +a_8^4x_{36}x_{54} +a_8^7a_{10}^3x_{36} -a_8^5a_{10}^2x_{36}y_{26}
                \\ && \qquad \mbox{}
        -a_8^3a_{10}x_{36}y_{26}^2 +a_8a_{10}^3x_{36}x_{48}, \\[5mm]
& y_{60}^2 & = a_4^2x_{54}y_{58} +a_4^3a_8a_{10}^3x_{48}y_{22}
        -a_4^2a_8^5a_{10}^3y_{20}y_{22} -a_4^2a_8^3a_{10}^3y_{58}
        -a_4^2a_8^2a_{10}^2y_{76}
                \\ && \qquad \mbox{}
        -a_4^2a_8a_{10}x_{54}y_{20}^2
        +a_4^2a_{10}^2x_{48}y_{22}^2 +a_4a_8^5a_{10}^4x_{36}
        -a_4a_8^2a_{10}x_{36}x_{54} -a_8^4a_{10}^3x_{36}y_{22}
                \\ && \qquad \mbox{}
        +a_8^3a_{10}^4x_{36}y_{20} +a_8^2a_{10}^2x_{36}y_{22}y_{26}
        +a_8a_{10}^3x_{36}y_{20}y_{26} +a_{10}^2x_{36}y_{64} , \\
& y_{60}y_{76} & = a_4^3x_{48}x_{54}y_{22}
        +a_4^4a_8a_{10}x_{48}x_{54} -a_4^3a_8^5a_{10}x_{54}y_{20}
        +a_4^3a_8^3a_{10}^3x_{48}y_{22}
        -a_4^2a_8^4x_{54}y_{20}y_{22}
                \\ && \qquad \mbox{}
        -a_4^2a_8^7a_{10}^3y_{20}y_{22} +a_4^2a_8^6a_{10}^2y_{60}
        -a_4^2a_8^5a_{10}^3y_{58} +a_4^2a_8^4a_{10}^2y_{76}
        -a_4^2a_8a_{10}^3x_{48}y_{20}y_{22} 
                \\ && \qquad \mbox{}
        +a_4^2a_{10}^2x_{48}y_{60}
        +a_4a_8^4a_{10}x_{36}x_{54} +a_4a_8a_{10}^4x_{36}x_{48}
        +a_8^3x_{36}x_{54}y_{22} 
                \\ && \qquad \mbox{} 
        +a_8^4a_{10}^2x_{36}y_{22}y_{26}
        +a_8^2a_{10}x_{36}x_{54}y_{20} -a_8^2a_{10}^2x_{36}y_{64}
        -a_{10}^3x_{36}x_{48}y_{22}, \\
& y_{60}y_{26} & = \mbox{} -a_4a_8x_{54}y_{20}
        -a_8^3a_{10}^2y_{20}y_{22} +a_8^2a_{10}y_{60} -a_8a_{10}^2y_{58}
        -a_{10}y_{76}, 
\end{array}$\\

$\begin{array}{rll}
                \hbox to 10mm{} & & \\[-6mm]
& y_{60}y_{64} & = a_4^4a_{10}x_{48}x_{54}
        +a_4^3a_8^4a_{10}x_{54}y_{20} +a_4^3a_8^2a_{10}^3x_{48}y_{22}
        -a_4^2a_8x_{54}y_{58}
                \\ && \qquad \mbox{}
        -a_4^2a_8^3x_{54}y_{20}y_{22}
        +a_4^2a_8^6a_{10}^3y_{20}y_{22} +a_4^2a_8^5a_{10}^2y_{60}
        +a_4^2a_8^4a_{10}^3y_{58} +a_4^2a_8a_{10}^2x_{48}y_{22}^2
                \\ && \qquad \mbox{}
        -a_4^2a_{10}^3x_{48}y_{20}y_{22} +a_4a_8^6a_{10}^4x_{36}
        -a_4a_8^3a_{10}x_{36}x_{54}
                \\ && \qquad \mbox{}
        -a_4a_{10}^4x_{36}x_{48}
        +a_8^5a_{10}^3x_{36}y_{22} -a_8^4a_{10}^4x_{36}y_{20}, \\
& y_{26}^3 & = a_8^3x_{54} -a_8^4a_{10}^2y_{26}
        -a_8^2a_{10}y_{26}^2 +a_{10}^3x_{48}, \\
& y_{64}y_{26} & = \mbox{} -a_4a_8^4x_{54} +a_8^2x_{54}y_{20}
        -a_8^6a_{10}^2y_{22} +a_8^4a_{10}y_{22}y_{26} +a_8^2a_{10}y_{64}
        +a_{10}^2x_{48}y_{22}, \\
& y_{64}^2 & = \mbox{} -a_4^2a_8^7a_{10}x_{54} +a_4^2a_8a_{10}x_{48}x_{54}
        -a_4a_8^6x_{54}y_{22} -a_4a_8^9a_{10}^3y_{22}
        +a_4a_8^5a_{10}x_{54}y_{20}
                \\ && \qquad \mbox{}
        +a_4a_8^3a_{10}^3x_{48}y_{22}
        +a_8^2x_{54}y_{58} -a_8^4x_{54}y_{20}y_{22}
        -a_8^7a_{10}^3y_{20}y_{22} +a_8^5a_{10}^3y_{58}
                \\ && \qquad \mbox{}
        -a_8^4a_{10}^2y_{76}
        -a_8^2a_{10}^2x_{48}y_{22}^2
        -a_8a_{10}^3x_{48}y_{20}y_{22} -a_{10}^2x_{48}y_{60}, \\
& y_{76}^2 &=  \mbox{} -a_4^2x_{48}x_{54}y_{20}y_{22}
        -a_4^4a_8^3a_{10}x_{48}x_{54} -a_4^3a_8^2x_{48}x_{54}y_{22}
        +a_4^3a_8^7a_{10}x_{54}y_{20}
                \\ && \qquad \mbox{}
        +a_4^3a_8^5a_{10}^3x_{48}y_{22} -a_4^3a_8a_{10}x_{48}x_{54}y_{20}
        -a_4^2a_8^4x_{54}y_{58} +a_4^2a_8^6x_{54}y_{20}y_{22}
                \\ && \qquad \mbox{}
        -a_4^2a_8^9a_{10}^3y_{20}y_{22} -a_4^2a_8^8a_{10}^2y_{60}
        -a_4^2a_8^7a_{10}^3y_{58} +a_4^2a_8^5a_{10}x_{54}y_{20}^2
                \\ && \qquad \mbox{}
        -a_4^2a_8^4a_{10}^2x_{48}y_{22}^2 -a_4^2a_8a_{10}^3x_{48}y_{58}
        +a_4^2a_{10}^2x_{48}y_{76} -a_4a_8^9a_{10}^4x_{36}
                \\ && \qquad \mbox{}
        -a_4a_8^3a_{10}^4x_{36}x_{48} +a_8^2x_{36}x_{54}y_{20}y_{26}
        -a_8^5x_{36}x_{54}y_{22} +a_8^8a_{10}^3x_{36}y_{22}
                \\ && \qquad \mbox{}
        -a_8^7a_{10}^4x_{36}y_{20} -a_8^5a_{10}^3x_{36}y_{20}y_{26}
        +a_8^4a_{10}^2x_{36}y_{64} -a_8^2a_{10}^3x_{36}x_{48}y_{22}
                \\ && \qquad \mbox{}
        +a_{10}^2x_{36}x_{48}y_{22}y_{26}, \\
& y_{76}y_{26} & = a_4a_8^3x_{54}y_{20} -a_4a_8a_{10}^2x_{48}y_{22} 
        +a_8x_{54}y_{20}^2 -a_8^5a_{10}^2y_{20}y_{22}
                \\ && \qquad \mbox{}
        +a_8^4a_{10}y_{60} -a_8^3a_{10}^2y_{58} -a_8^2a_{10}y_{76}
        +a_{10}x_{48}y_{22}^2, \\
& y_{64}y_{76} & = \mbox{} -a_4^3a_8^2a_{10}x_{48}x_{54}
        -a_4^2a_8x_{48}x_{54}y_{22} -a_4^2a_8^6a_{10}x_{54}y_{20}
        -a_4^2a_{10}x_{48}x_{54}y_{20}
                \\ && \qquad \mbox{}
        -a_4a_8^3x_{54}y_{58} -a_4a_8^5x_{54}y_{20}y_{22}
        +a_4a_8^8a_{10}^3y_{20}y_{22} +a_4a_8^7a_{10}^2y_{60}
                \\ && \qquad \mbox{}
        +a_4a_8^6a_{10}^3y_{58} -a_4a_8^5a_{10}^2y_{76}
        -a_4a_8^4a_{10}x_{54}y_{20}^2 -a_4a_8^3a_{10}^2x_{48}y_{22}^2
                \\ && \qquad \mbox{}
        +a_4a_8a_{10}^2x_{48}y_{60} -a_4a_{10}^3x_{48}y_{58}
        +a_8^3x_{36}x_{54}y_{26} +a_8^8a_{10}^4x_{36}
                \\ && \qquad \mbox{}
        -a_8^6a_{10}^3x_{36}y_{26} +a_8^5a_{10}x_{36}x_{54}
        +a_8^5a_{10}^2y_{58}y_{22} +a_8^4a_{10}^2x_{36}y_{26}^2
                \\ && \qquad \mbox{}
        +a_8^2a_{10}^4x_{36}x_{48} +a_{10}^3x_{36}x_{48}y_{26}.
\end{array}$

\renewcommand{\arraystretch}{1}
$\begin{array}{ll}
                \hbox to 10mm{} &  \\[-6mm]
\rm{\hphantom{i}ii)}  
     & a_9^2, \ y_{21}^2, \ y_{25}^2, \ y_{27}^2,\\
     &   a_9y_{21} + x_{26}a_4,
       \ a_9y_{25} + x_{26}a_8,
       \ a_9y_{27} + x_{26}a_{10},\\
     &   y_{21}y_{25} + x_{26}y_{20},
       \ y_{21}y_{27} - x_{26}y_{22},
       \ y_{25}y_{27} - x_{26}y_{26},\\
\end{array}$

$\begin{array}{ll}
                \hbox to 10mm{} &  \\[-6mm]
\rm{iii)} & a_9\partial^2Q, \ y_{21}\partial^2Q, 
        \ y_{25}\partial^2Q, \ y_{27}\partial^2Q, 
        \ x_{26}\partial^2Q,\\
     & a_9a_4, \ a_9a_8, \ a_9a_{10},\\
     & y_{21}a_4, \ y_{25}a_8, \ y_{27}a_{10},\\
     & y_{21}a_8 +a_9y_{20} = y_{25}a_4 - a_9y_{20},\\
     & y_{21}a_{10} - a_9y_{22} = y_{27}a_4 +a_9y_{22},\\
     & y_{25}a_{10} - a_9y_{26} = y_{27}a_8 +a_9y_{26},\\
     & y_{21}y_{20}, \ y_{25}y_{20}, \ y_{21}y_{22}, 
     \ y_{27}y_{22}, \ y_{25}y_{26}, \ y_{27}y_{26},\\
     & y_{27}y_{20} - y_{25}y_{22} 
          = y_{25}y_{22} + y_{21}y_{26} 
          = - y_{27}y_{20} - y_{21}y_{26}.
\end{array}$
~\\

Here the definition of $\partial$ is given in {\rm (\ref{eq:33})}, 
and the table of $\partial^2$-image is given in Lemma \ref{42}.
\section{Proof of Theorem \ref{21}} 

First let us recall some results from \cite{MS}. 
We consider the differential algebra $(\overline{V},d)$ defined 
as follows:
$$\overline{V} = T(a_4,a_8,a_{10},b_{12},b_{16},b_{18},a_9,c_{17})/J,$$
where $J$ is the ideal generated by 
$$a_9b_j - b_j a_9 + c_{17}a_{j-8} \quad \mbox{for}\quad j = 12,16,18$$
and
$$\alpha\beta - (-1)^{|\beta||\alpha|} \beta\alpha \hskip3em 
        \begin{tabular}[t]{l}
           for all the pairs $(\alpha,\beta)$  of the generators \\
           except \ $(a_9,b_j)$ \ with \ $j = 12,16,18$ 
              and \ $(a_9,c_{17})$;
        \end{tabular}
$$
the differential $d$ is given over generators by 
\begin{equation}
 da_i    = 0            \quad \mbox{for } \ i = 4,8,9,10, \quad 
   dc_{17} = a_9^2,       \quad 
   db_j    = -a_9 a_{j-8} \quad \mbox{for }\ j = 12,16,18.
\label{eq:30}
\end{equation}
Then the
differential algebra $(\overline{V},d)$ is seen to be a resolution 
of $\ZZ_3$ over $H^*(E_6)$. 
(For a proof, see Theorem 5.7 and Corollary 5.8 in \cite{MS}.) 
So by the definition of Cotor we have\\
\\
{\bf Theorem} (\cite{MS})\qquad
$H^*(\overline{V},d) = \Cotor_{H^*(E_6)}^{}(\ZZ_3,\ZZ_3).$

\bigskip
We will use a spectral sequence to calculate $H^*(\overline{V},d)$.
To that end we need to introduce a filtration in $\overline{V}$. 
To begin with, we fix an additive basis of $\overline{V}$. 
The following lemma is straightforward 
from the definition of the ideal $J$.

\begin{Lemma}\label{31}
{\it We have an additive isomorphism }
$$\overline{V} \cong T(a_9,c_{17}) 
        \otimes \ZZ_3[a_4,a_8,a_{10},b_{12},b_{16},b_{18}].$$
\end{Lemma}
 
Hence as a basis of $\overline{V}$ we can choose 
the following monomials:
\begin{equation}
a_9^{\ell_1}c_{17}^{m_1} \ldots a_9^{\ell_p}c_{17}^{m_p}
        a_4^{i_1} a_8^{i_2} a_{10}^{i_3}
        b_{12}^{j_1}b_{16}^{j_2}b_{18}^{j_3}. \label{eq:32}
\end{equation}

\indent 
We define a weight by
$$ \begin{array}{l}
   w(a_i) = 2 \qquad \mbox{ for }\ i = 4,8,10; \\
   w(b_j) = 0 \qquad \mbox{ for }\ j = 12,16,18; \\
   w(a_9) = 1, \quad w(c_{17}) = 2, \quad w(0) = \infty.
\end{array}$$
For a monomial of the form (\ref{eq:32}), 
we define the weight by
$$w(\alpha_1 \ldots \alpha_n) = w(\alpha_1) + \cdots + w(\alpha_n).$$
For an element $x = \sum \lambda_i x_i$, 
where $x_i$ is a monomial given in (\ref{eq:32}) 
and $\lambda_i \in \ZZ_3$, 
we define that $w(x) = \displaystyle \inf_i w(x_i)$.
Then we introduce a filtration in  $\overline{V}$ by setting
$$F^p\overline{V} 
        = \{x \in \overline{V} 
                \mid w(x) \ge p\} \quad \mbox{for }\ p \in \NN.$$
It is easily seen that $\{F^p\overline{V}\}$ is a family 
of differential graded algebras giving a decreasing filtration.
Hence from this filtration we can construct 
a spectral sequence $\{E_r,d_r\}$ such 
that $E_0=\overline{V}$ and $E_{\infty}=H^*(\overline{V},d)$, 
namely, converging to $\Cotor_{H^*(E_6)}^{}(\ZZ_3,\ZZ_3)$.

It is convenient for later use to introduce here 
an auxiliary derivation $\partial$ on a subalgebra 
        \ $\ZZ_3[a_4,a_8,a_{10}, b_{12},b_{16},b_{18}]$ of 
        \ $\overline{V}$ defined as in (5.9) of \cite{MS}:
\begin{equation}
  \partial a_i = 0        \quad \mbox{for }\ i = 4,8,10;\quad 
    \partial b_j = -a_{j-8} \quad \mbox{for }\ j = 12,16,18.
                    \label{eq:33} 
\end{equation}

Since the only non trivial differential on the generators of $E_0$ is 
$d_0(c_{17}) = a_9^2$ by (\ref{eq:30}), 
we have that
$$E_1 = \ZZ_3[x_{26}] \otimes \ZZ_3[a_4,a_8,a_{10}] 
        \otimes \Lambda(a_9) \otimes \ZZ_3[b_{12},b_{16},b_{18}],$$
where $x_{26} = a_9 c_{17} - c_{17} a_9$. 
Since $d_1 = d_2 = 0$ for degree reasons, we have 
$$E_1 = E_2 = E_3.$$

We remark here that $E_3 = E_1$ is a commutative algebra 
in the graded sense and also that $d_3$ is 
a derivative differential, i.e.,
$$d_3(xy) = d_3(x)y \pm xd_3(y), \qquad x,y \in E_3$$
such that 
$$ d_3(b_{12}) = -a_4a_9, \quad d_3(b_{16}) = -a_8a_9, 
                          \quad d_3(b_{18}) = -a_{10}a_9, $$
which are obtained by the differential formula (\ref{eq:30}).

In the next steps, we will use the following 

{\bf Convention} \quad
Let $(M,d)$ be a differential module such that 
\ $M = M_1 \oplus M_2 \oplus N_1$ \ and that 
\ $d(M_1) \subset N_1$ 
\ and $d|{}_{M_2 \oplus N_1} = 0$. 
Then we denote ${\rm Ker}\; d \cap M_1$ \ by \ ${\rm Ker}\; d$ 
\ and \ $N_1/d(M_1)$ \ by \ $\Coker d$. 

\bigskip
Under our convention, we have that 
$$H(M) = {\rm Ker}\; d \oplus \Coker d \oplus M_2.$$
Moreover, for simplicity we will describe the differential $d$ only 
on $M_1$. If not specifically denoted, 
we understand that $d(a) = 0$ for $a \in M_2 \oplus N_1$.

To calculate the $E_4$-term, 
we introduce another spectral sequence starting with  
the $E_3$-term. That is, we define a weight $\wtilde{w}$ by
\begin{equation}
\wtilde{w}(b_i) = 0 \quad \mbox{for}\ i = 12,16,18, 
\ \wtilde{w}(x_{26}) = \wtilde{w}(a_9) = \wtilde{w}(a_4) = 0,
\ \wtilde{w}(a_8) =  1,
\ \wtilde{w}(a_{10}) = 3. \label{eq:34}
\end{equation}
By similarly introducing a filtration 
using the weight $\wtilde{w}$ as before, we obtain 
a new spectral sequence $\{\wtilde{E}_r,\wtilde{d}_r\}$ such
that $\wtilde{E}_0 = E_3$ and $\wtilde{E}_{\infty} = E_4$.

At the $\wtilde{E}_0$-term we see that 
the relation $d_3(b_{12}) = -a_4a_9$ gives %
$\wtilde{d}_0(b_{12}) = -a_4a_9$,
from which we obtain that 
$$\wtilde{E}_1 
        = \Big(\ZZ_3[a_4,a_8,a_{10},x_{26}]\{1, a_9b_{12}^2\} 
            \oplus \ZZ_3[a_8,a_{10},x_{26}]\{a_9, a_9b_{12}\}\Big) 
         \otimes \ZZ_3[b_{16},b_{18}] \otimes \ZZ_3[b_{12}^3].$$

Since the relation $d_3(b_{16}) = -a_8a_9$ gives %
$\wtilde{d}_1(b_{16}) = -a_9a_8$, we have that 
\begin{eqnarray*}
\wtilde{E}_2 
& = & \Big( \ZZ_3[a_4,a_8,a_{10},x_{26}] 
        \oplus \ZZ_3[a_4,a_8,a_{10},x_{26}]\{a_4b_{16}, a_4b_{16}^2\} \\
& & \mbox{}\quad 
        \oplus \ZZ_3[a_{10},x_{26}]\{a_9, a_9b_{16}\} 
        \oplus \ZZ_3[a_8,a_{10},x_{26}]\{a_9b_{16}^2\} \\
& & \mbox{}\quad 
        \oplus \ZZ_3[a_4,a_8,a_{10},x_{26}]\{a_9b_{12}^2,
                a_9b_{12}^2b_{16}, a_9b_{12}^2b_{16}^2\} \\
& & \mbox{}\quad 
        \oplus \ZZ_3[a_8,a_{10},x_{26}]\{a_9b_{12},
                a_9b_{12}b_{16}, a_9b_{12}b_{16}^2\} \,
      \Big) \otimes \ZZ_3[b_{18}] \otimes \ZZ_3[b_{12}^3,b_{16}^3].
\end{eqnarray*}

It follows from the relation %
$d(a_4b_{16}^2 + a_8b_{12}b_{16}) 
        = -a_8^2 a_9b_{12} + a_4a_8^2c_{17}$ giving %
$d_3(a_4b_{16}^2 + a_8b_{12}b_{16}) = -a_8^2 a_9b_{12}$ in $E_3$  
that $\wtilde{d}_2(a_4b_{16}^2) = -a_8^2 a_9b_{12}$, 
which implies that 
\begin{eqnarray*}
\wtilde{E}_3 
& = & \Big(\ZZ_3[a_4,a_8,a_{10},x_{26}]\{1, a_4b_{16}, a_4^2b_{16}^2\} 
                        \oplus \ZZ_3[a_{10},x_{26}]\{a_9, a_9b_{16}\}\\
&   & \mbox{}\qquad 
        \oplus \ZZ_3[a_8,a_{10},x_{26}]\{a_9b_{16}^2\} 
        \oplus \ZZ_3[a_4,a_8,a_{10},x_{26}]
            \{a_9b_{12}^2, a_9b_{12}^2b_{16}, a_9b_{12}^2b_{16}^2\}\\
&   & \mbox{}\qquad 
        \oplus \ZZ_3[a_{10},x_{26}] \otimes \Delta(a_8)\{a_9b_{12}\} 
        \oplus \ZZ_3[a_8,a_{10},x_{26}]
                \{a_9b_{12}b_{16}, a_9b_{12}b_{16}^2\}
      \Big)\\
&   & \mbox{}\quad 
        \otimes \ZZ_3[b_{12}^3,b_{16}^3]\otimes \ZZ_3[b_{18}].
\end{eqnarray*}

Using the expression 
$\ZZ_3[b_{18}]=\ZZ_3\{1,b_{18},b_{18}^2\} \otimes \ZZ_3[b_{18}^3]$, 
the $\wtilde{E}_3$-term may be rewritten as

\mathvskip \mbox{}\hskip2em
$\wtilde{E}_3 = \Big(
\ZZ_3[a_4,a_8,a_{10},x_{26}]
        \{     1,                     b_{18},                b_{18}^2, 
            a_4b_{16},     a_4b_{16}  b_{18},    a_4b_{16}   b_{18}^2, 
          a_4^2b_{16}^2, a_4^2b_{16}^2b_{18}, a_4^2b_{16}^2b_{18}^2
        \}$

$\mbox{}\hskip5em 
\oplus \ZZ_3[a_{10},x_{26}]
        \{a_9,       a_9      b_{18}, a_9      b_{18}^2,
          a_9b_{16}, a_9b_{16}b_{18}, a_9b_{16}b_{18}^2\}$

$\mbox{}\hskip5em 
\oplus \ZZ_3[a_8,a_{10},x_{26}]
        \{a_9b_{16}^2, a_9b_{16}^2b_{18}, a_9b_{16}^2b_{18}^2
        \}$                                     \mbox{}\quad\hfill (4)

$\mbox{}\hskip5em 
\oplus \ZZ_3[a_4,a_8,a_{10},x_{26}]
        \{a_9b_{12}^2,      a_9b_{12}^2b_{18}, a_9b_{12}^2b_{18}^2,
        a_9b_{12}^2b_{16},  a_9b_{12}^2b_{16}b_{18}, 
                                        a_9b_{12}^2b_{16}b_{18}^2,$

$\mbox{}\hskip20em 
        a_9b_{12}^2b_{16}^2, a_9b_{12}^2b_{16}^2b_{18}, 
                                        a_9b_{12}^2b_{16}^2b_{18}^2
        \}$                                     \mbox{}\quad\hfill (5)

$\mbox{}\hskip5em 
\oplus \ZZ_3[a_{10},x_{26}] \otimes \Delta(a_8)
                \{a_9b_{12}, a_9b_{12}b_{18}, a_9b_{12}b_{18}^2\}$

$\mbox{}\hskip5em 
\oplus \ZZ_3[a_8,a_{10},x_{26}]
        \{a_9b_{12}b_{16}, a_9b_{12}b_{16}b_{18}, a_9b_{12}b_{16}b_{18}^2,
          $

$\mbox{}\hskip20em 
          a_9b_{12}b_{16}^2, a_9b_{12}b_{16}^2b_{18},
                                             a_9b_{12}b_{16}^2b_{18}^2
        \} 
\Big)$                                          \mbox{}\quad\hfill (7)

$\mbox{}\hskip4em\otimes \ZZ_3[b_{12}^3,b_{16}^3,b_{18}^3]$,

\bigskip\noindent
where (4), (5) and (7) will appear in Lemma 3.6. 
The only non trivial differential 
$$\wtilde{d}_3 : \ZZ_3[a_4, a_8,a_{10},x_{26}]\{b_{18},b_{18}^2\} 
        \longrightarrow \ZZ_3[a_{10},x_{26}]\{a_9, a_9b_{18}\}$$
is given by \ $\wtilde{d}_3(b_{18}) = -a_9a_{10}$ %
which is obtained from $d(b_{18}) = -a_9a_{10}$,  and hence %
$\wtilde{d}_3(b_{18}^2) = a_9a_{10}b_{18}$.
It is immediate to see that 

\mathvskip
\mbox{}\hskip3em  
$\KER \wtilde{d}_3 
        = \ZZ_3[a_4,a_8,a_{10},x_{26}]\{a_4b_{18}, a_4b_{18}^2\} 
        \oplus \ZZ_3[a_8,a_{10},x_{26}]\{a_8b_{18}, a_8b_{18}^2\}$, 
                                        \mbox{}\quad\hfill $\b{1}'$

\mbox{}\hskip3em  
$\Coker \wtilde{d}_3 = \ZZ_3[x_{26}]\{a_9, a_9b_{18}\}$.
                                        \mbox{}\quad\hfill $\b{2}'$

Thus we have that 
$$\begin{array}{ll}
\wtilde{E}_4 = & 
\Big(
  \ZZ_3[a_4,a_8,a_{10},x_{26}]\{a_4b_{18}, a_4b_{18}^2\}
  \oplus \ZZ_3[a_8,a_{10},x_{26}]\{a_8b_{18}, a_8b_{18}^2\} 
                                                \\ & \mbox{}\quad
  \oplus \ZZ_3[a_4,a_8,a_{10},x_{26}]
        \{1,  a_4b_{16},     a_4b_{16}b_{18},     a_4b_{16}b_{18}^2, 
          a_4^2b_{16}^2, a_4^2b_{16}^2b_{18}, a_4^2b_{16}^2b_{18}^2\}
                                                \\ & \mbox{}\quad
  \oplus \ZZ_3[x_{26}]\{a_9, a_9b_{18}\}        \\ & \mbox{}\quad
  \oplus \ZZ_3[a_{10},x_{26}]
        \{a_9b_{18}^2, a_9b_{16}, a_9b_{16}b_{18}, a_9b_{16}b_{18}^2\}
                                                \\ & \mbox{}\quad
  \oplus \ZZ_3[a_8,a_{10},x_{26}]
        \{a_9b_{16}^2, a_9b_{16}^2b_{18}, a_9b_{16}^2b_{18}^2\} 
                                                \\ & \mbox{}\quad
  \oplus \ZZ_3[a_4,a_8, a_{10}, x_{26}]
        \{a_9b_{12}^2,      a_9b_{12}^2b_{18}, a_9b_{12}^2b_{18}^2,
        a_9b_{12}^2b_{16},  a_9b_{12}^2b_{16}b_{18}, 
                                        a_9b_{12}^2b_{16}b_{18}^2, 
                                                \\ & \mbox{}\quad
\hfill 
        a_9b_{12}^2b_{16}^2, a_9b_{12}^2b_{16}^2b_{18}, 
                                a_9b_{12}^2b_{16}^2b_{18}^2\}
                                                \\ & \mbox{}\quad
  \oplus \ZZ_3[a_{10},x_{26}] \otimes \Delta(a_8)
                \{a_9b_{12}, a_9b_{12}b_{18}, a_9b_{12}b_{18}^2\} 
                                                \\ & \mbox{}\quad
  \oplus \ZZ_3[a_8,a_{10},x_{26}]
        \{a_9b_{12}b_{16}, a_9b_{12}b_{16}b_{18}, a_9b_{12}b_{16}b_{18}^2,
                                                \\ & \mbox{}\quad
\hfill 
          a_9b_{12}b_{16}^2, a_9b_{12}b_{16}^2b_{18},
                                     a_9b_{12}b_{16}^2b_{18}^2\} 
\Big)                                           \\ & \mbox{}\quad
\mbox{}\hskip-1em
\otimes \ZZ_3[b_{12}^3,b_{16}^3,b_{18}^3].
\end{array}$$

\mathvskip
The relation  
$d_3(a_4b_{16}b_{18} - a_8b_{12}b_{18} 
                - a_{10}b_{12}b_{16}) = -a_8 a_9a_{10}b_{12}$ %
in $E_3$ implies that the only non trivial differential
$$\wtilde{d}_4 : \ZZ_3[a_4,a_8,a_{10},x_{26}]\{a_4b_{16}b_{18},
        a_4b_{16}b_{18}^2\} \rightarrow \ZZ_3[a_{10},x_{26}] 
        \otimes \Delta (a_8)\{a_9b_{12}, a_9b_{12}b_{18}\}$$
is given by \ $\wtilde{d}_4(a_4b_{16}b_{18}) 
= -a_8 a_9a_{10}b_{12}$, and hence $\wtilde{d}_4(a_4b_{16}b_{18}^2) 
= a_8 a_9a_{10}b_{12}b_{18}$. So we have that 

\mathvskip
\mbox{}\hskip3em\hskip-\arraycolsep
$\begin{array}{ll}
\KER \wtilde{d}_4 
        = \ZZ_3[a_4,a_8,a_{10},x_{26}]
                \{a_4^2b_{16}b_{18}, a_4^2b_{16}b_{18}^2\} \\
\mbox{}\hskip15em 
           \oplus \ZZ_3[a_8,a_{10},x_{26}]\{a_4a_8b_{16}b_{18},
                        a_4a_8b_{16}b_{18}^2\}, 
 \end{array}$                           \mbox{}\quad\hfill $\b{3}'$

\smallskip
\mbox{}\hskip3em  
$\Coker \wtilde{d}_4 
        = \ZZ_3[a_{10},x_{26}]\{a_9b_{12}, a_9b_{12}b_{18}\} 
          \oplus \ZZ_3[x_{26}]\{a_9a_8b_{12}, a_9a_8b_{12}b_{18}\}$.
                                        \mbox{}\quad\hfill $\b{4}'$

\mathvskip
Thus we have that 
$$\begin{array}{ll}
\wtilde{E}_5 = & 
\Big(
  \ZZ_3[a_4,a_8,a_{10},x_{26}]\{a_4b_{18}, a_4b_{18}^2\}
  \oplus \ZZ_3[a_8,a_{10},x_{26}]\{a_8b_{18}, a_8b_{18}^2\} 
                                                \\ & \mbox{}\quad
  \oplus \ZZ_3[a_4,a_8,a_{10},x_{26}]
        \{1,  a_4b_{16}, a_4^2b_{16}^2, 
                a_4^2b_{16}^2b_{18}, a_4^2b_{16}^2b_{18}^2\}
                                                \\ & \mbox{}\quad
  \oplus \ZZ_3[a_4,a_8,a_{10},x_{26}]
        \{ a_4^2b_{16}b_{18}, a_4^2b_{16}b_{18}^2 \}
  \oplus \ZZ_3[a_8,a_{10},x_{26}]
        \{ a_4a_8b_{16}b_{18}^{}, a_4a_8b_{16}b_{18}^2 \}
                                                \\ & \mbox{}\quad
  \oplus \ZZ_3[x_{26}]\{a_9, a_9b_{18}\}
  \oplus \ZZ_3[a_{10},x_{26}]
        \{a_9b_{18}^2, a_9b_{16}, a_9b_{16}b_{18}, a_9b_{16}b_{18}^2\}
                                                \\ & \mbox{}\quad
  \oplus \ZZ_3[a_8,a_{10},x_{26}]
        \{a_9b_{16}^2, a_9b_{16}^2b_{18}, a_9b_{16}^2b_{18}^2\} 
                                                \\ & \mbox{}\quad
  \oplus \ZZ_3[a_4,a_8, a_{10}, x_{26}]
        \{a_9b_{12}^2,      a_9b_{12}^2b_{18}, a_9b_{12}^2b_{18}^2,
        a_9b_{12}^2b_{16},  a_9b_{12}^2b_{16}b_{18}, 
                                        a_9b_{12}^2b_{16}b_{18}^2, 
                                                \\ & \mbox{}\quad
\hfill 
        a_9b_{12}^2b_{16}^2, a_9b_{12}^2b_{16}^2b_{18}, 
                                a_9b_{12}^2b_{16}^2b_{18}^2\}
                                                \\ & \mbox{}\quad
  \oplus \ZZ_3[a_{10},x_{26}]\{a_9b_{12}, a_9b_{12}b_{18} \}
  \oplus \ZZ_3[x_{26}]\{a_9a_8b_{12}, a_9a_8b_{12}b_{18}\}
                                                \\ & \mbox{}\quad
  \oplus \ZZ_3[a_{10},x_{26}] \otimes \Delta(a_8)\{a_9b_{12}b_{18}^2\} 
                                                \\ & \mbox{}\quad
  \oplus \ZZ_3[a_8,a_{10},x_{26}]
        \{a_9b_{12}b_{16}, a_9b_{12}b_{16}b_{18},
                                        a_9b_{12}b_{16}b_{18}^2,
                                                \\ & \mbox{}\quad
\hfill 
          a_9b_{12}b_{16}^2, a_9b_{12}b_{16}^2b_{18},
                                     a_9b_{12}b_{16}^2b_{18}^2\} 
\Big)                                           \\ & \mbox{}\quad
\mbox{}\hskip-1em
\otimes \ZZ_3[b_{12}^3,b_{16}^3,b_{18}^3].
\end{array}$$

\mathvskip
The relation %
$d_3(a_8b_{18}^2 +a_{10}b_{16}b_{18}) = -a_{10}^2 a_9b_{16}$ in $E_3$ %
implies that the only non trivial differential 
$$\wtilde{d}_5:\ZZ_3[a_8,a_{10},x_{26}]\{a_8b_{18}^2\} \\
        \rightarrow \ZZ_3[a_{10},x_{26}]\{a_9b_{16}\}$$
is given by 
\ $\wtilde{d}_5(a_8b_{18}^2) = -a_{10}^2 a_9b_{16}$.
Hence we obtain that 

\mathvskip
\mbox{}\hskip3em  $\KER \wtilde{d}_5 
= \ZZ_3[a_8,a_{10},x_{26}]\{a_8^2b_{18}^2\}$,
\mbox{}\quad\hfill $\b{5}'$

\mbox{}\hskip3em  $\Coker \wtilde{d}_5 = \Delta (a_{10}) \otimes 
\ZZ_3[x_{26}]\{a_9b_{16}\}$. \mbox{}\quad\hfill $\b{6}'$

\mathvskip
Thus we have that 
$$\begin{array}{ll}
\wtilde{E}_6 = & 
\Big(
  \ZZ_3[a_4,a_8,a_{10},x_{26}]\{a_4b_{18}, a_4b_{18}^2\}
  \oplus \ZZ_3[a_8,a_{10},x_{26}]\{a_8b_{18}, a_8^2b_{18}^2 \} 
                                                \\ & \mbox{}\quad
  \oplus \ZZ_3[a_4,a_8,a_{10},x_{26}]
        \{1,  a_4b_{16}, a_4^2b_{16}^2, 
                a_4^2b_{16}^2b_{18}, a_4^2b_{16}^2b_{18}^2,
                a_4^2b_{16}b_{18}, a_4^2b_{16}b_{18}^2\} 
                                                \\ & \mbox{}\quad
  \oplus \ZZ_3[a_8,a_{10},x_{26}]
        \{a_4a_8b_{16}b_{18}, a_4a_8b_{16}b_{18}^2\}
                                                \\ & \mbox{}\quad
  \oplus \ZZ_3[x_{26}]\{a_9, a_9b_{18}, 
                        a_9a_8b_{12}, a_9a_8b_{12}b_{18}\} 
                                                \\ & \mbox{}\quad
  \oplus \ZZ_3[a_{10},x_{26}]
        \{a_9b_{18}^2, a_9b_{16}b_{18}, a_9b_{16}b_{18}^2,
                        a_9b_{12}, a_9b_{12}b_{18}\} 
                                                \\ & \mbox{}\quad
  \oplus \ZZ_3[x_{26}] \otimes \Delta(a_{10})\{a_9b_{16}\} 
                                                \\ & \mbox{}\quad
  \oplus \ZZ_3[a_8,a_{10},x_{26}]
        \{a_9b_{16}^2, a_9b_{16}^2b_{18}, a_9b_{16}^2b_{18}^2\} 
                                                \\ & \mbox{}\quad
  \oplus \ZZ_3[a_4,a_8, a_{10}, x_{26}]
        \{a_9b_{12}^2,      a_9b_{12}^2b_{18}, a_9b_{12}^2b_{18}^2,
        a_9b_{12}^2b_{16},  a_9b_{12}^2b_{16}b_{18}, 
                                        a_9b_{12}^2b_{16}b_{18}^2, 
                                                \\ & \mbox{}\quad
\hfill 
        a_9b_{12}^2b_{16}^2, a_9b_{12}^2b_{16}^2b_{18}, 
                                a_9b_{12}^2b_{16}^2b_{18}^2\}
                                                \\ & \mbox{}\quad
  \oplus \ZZ_3[a_{10},x_{26}] \otimes \Delta(a_8)\{a_9b_{12}b_{18}^2\} 
                                                \\ & \mbox{}\quad
  \oplus \ZZ_3[a_8,a_{10},x_{26}]
        \{a_9b_{12}b_{16}, a_9b_{12}b_{16}b_{18},
                                        a_9b_{12}b_{16}b_{18}^2,
                                                \\ & \mbox{}\quad
\hfill 
          a_9b_{12}b_{16}^2, a_9b_{12}b_{16}^2b_{18},
                                     a_9b_{12}b_{16}^2b_{18}^2\} 
\Big)                                           \\ & \mbox{}\quad
\mbox{}\hskip-1em
\otimes \ZZ_3[b_{12}^3,b_{16}^3,b_{18}^3].
\end{array}$$

Similarly the relation %
$d_3(a_4b_{18}^2) = a_9a_{10}^2b_{12}$ in $E_3$ %
implies that the only non trivial differential
   $$\wtilde{d}_6:\ZZ_3[a_4, a_8,a_{10},x_{26}]\{a_4b_{18}^2\} 
        \rightarrow \ZZ_3[a_{10},x_{26}]\{a_9b_{12}\}$$
is given by \ $\wtilde{d}_6(a_4b_{18}^2) = a_9a_{10}^2b_{12}$, %
from which we have that 

\mathvskip
\mbox{}\hskip3em  
$\KER \wtilde{d}_6 = \ZZ_3[a_4,a_8,a_{10},x_{26}]\{a_4^2b_{18}^2\}
        \oplus \ZZ_3[a_8,a_{10},x_{26}]
                \{a_4a_8b_{18}^2\}$,    \mbox{}\quad\hfill $\b{7}'$

\mbox{}\hskip3em  
$\Coker \wtilde{d}_6 
        = \ZZ_3[x_{26}] \otimes \Delta (a_{10})\{a_9b_{12}\}$. 
                                        \mbox{}\quad\hfill $\b{8}'$

\mathvskip
Thus we have that 
$$\begin{array}{ll}
\wtilde{E}_7 = & 
\Big(
  \ZZ_3[a_4,a_8,a_{10},x_{26}]\{a_4b_{18}, a_4^2b_{18}^2 \} 
  \oplus \ZZ_3[a_8,a_{10},x_{26}]
                \{a_8b_{18}, a_8b_{18}^2, a_4a_8b_{18}^2 \} 
                                                \\ & \mbox{}\quad
  \oplus \ZZ_3[a_4,a_8,a_{10},x_{26}]
        \{1,  a_4b_{16}, a_4^2b_{16}^2, 
                a_4^2b_{16}^2b_{18}, a_4^2b_{16}^2b_{18}^2,
                a_4^2b_{16}b_{18}, a_4^2b_{16}b_{18}^2\} 
                                                \\ & \mbox{}\quad
  \oplus \ZZ_3[a_8,a_{10},x_{26}]
        \{a_4a_8b_{16}b_{18}, a_4a_8b_{16}b_{18}^2\}
                                                \\ & \mbox{}\quad
  \oplus \ZZ_3[x_{26}]\{a_9, a_9b_{18}, 
                        a_9a_8b_{12}, a_9a_8b_{12}b_{18}\} 
                                                \\ & \mbox{}\quad
  \oplus \ZZ_3[a_{10},x_{26}]
        \{a_9b_{18}^2, a_9b_{16}b_{18}, a_9b_{16}b_{18}^2,
                        a_9b_{12}b_{18}\} 
                                                \\ & \mbox{}\quad
  \oplus \ZZ_3[x_{26}] \otimes \Delta(a_{10})
        \{a_9b_{12}, a_9b_{16}\} 
                                                \\ & \mbox{}\quad
  \oplus \ZZ_3[a_8,a_{10},x_{26}]
        \{a_9b_{16}^2, a_9b_{16}^2b_{18}, a_9b_{16}^2b_{18}^2\} 
                                                \\ & \mbox{}\quad
  \oplus \ZZ_3[a_4,a_8, a_{10}, x_{26}]
        \{a_9b_{12}^2,      a_9b_{12}^2b_{18}, a_9b_{12}^2b_{18}^2,
        a_9b_{12}^2b_{16},  a_9b_{12}^2b_{16}b_{18}, 
                                        a_9b_{12}^2b_{16}b_{18}^2, 
                                                \\ & \mbox{}\quad
\hfill 
        a_9b_{12}^2b_{16}^2, a_9b_{12}^2b_{16}^2b_{18}, 
                                a_9b_{12}^2b_{16}^2b_{18}^2\}
                                                \\ & \mbox{}\quad
  \oplus \ZZ_3[a_{10},x_{26}] \otimes \Delta(a_8)\{a_9b_{12}b_{18}^2\} 
                                                \\ & \mbox{}\quad
  \oplus \ZZ_3[a_8,a_{10},x_{26}]
        \{a_9b_{12}b_{16}, a_9b_{12}b_{16}b_{18},
                                        a_9b_{12}b_{16}b_{18}^2,
                                                \\ & \mbox{}\quad
\hfill 
          a_9b_{12}b_{16}^2, a_9b_{12}b_{16}^2b_{18},
                                     a_9b_{12}b_{16}^2b_{18}^2\} 
\Big)                                           \\ & \mbox{}\quad
\mbox{}\hskip-1em
\otimes \ZZ_3[b_{12}^3,b_{16}^3,b_{18}^3].
\end{array}$$

\mathvskip
Then for degree reasons we see 
that $\wtilde{d}_r = 0$ for $r \ge 7$ so that 
\ $\wtilde{E}_7 = \wtilde{E}_\infty$,
which gives the $E_4$-term of the original spectral sequence.

Now we can summarize our calculations 
from (4) to (7) and from $\b{1}'$ to $\b{8}'$ as follows; 

\noindent
for the kernel parts, we have

\mathvskip
\mbox{}\hskip3em 
$\ZZ_3[a_4,a_8,a_{10},x_{26}]
        \{1, a_4b_{16}, a_4^2b_{16}^2, 
             a_4b_{18}, a_4^2b_{18}^2, 
             a_4^2b_{16}b_{18}, a_4^2b_{16}^2b_{18},
             a_4^2b_{16}b_{18}^2, a_4^2b_{16}^2b_{18}^2\} \\
\mbox{}\hskip7em 
  \oplus \ZZ_3[a_8,a_{10},x_{26}]
        \{a_8b_{16}, a_8^2b_{16}^2, a_4a_8b_{16}b_{18}, 
                a_4a_8b_{18}^2, a_4a_8b_{16}b_{18}^2\};$
                                                \mbox{}\quad\hfill (1)

\mathvskip\noindent
for the cokernel parts, we have 

\mathvskip
\mbox{}\hskip3em 
$\ZZ_3[x_{26}]\{a_9, a_9b_{12}, a_9b_{16}, a_9b_{18}, a_9a_8b_{12},
        a_9a_{10}b_{12}, a_9a_{10}b_{16}, a_9a_8b_{12}b_{18}\}$ 
                                                \mbox{}\quad\hfill (2)

$\mbox{}\hskip7em \oplus \ZZ_3[a_{10},x_{26}]\{a_9b_{12}b_{18}\}$.

\mathvskip\noindent
Hence we can write

\mathvskip\noindent
\begin{equation}
\wtilde{E}_7 = \wtilde{E}_\infty 
  = (1) \oplus (2) 
        \oplus \ZZ_3[a_{10},x_{26}]\{a_9b_{12}b_{18}\} 
        \oplus \ZZ_3[a_{10},x_{26}]
                \{a_9b_{16}b_{18}, a_9b_{16}b_{18}^2, a_9b_{18}^2\}
\label{eq:35}
\end{equation}
$\mbox{}\hskip17em 
        \oplus (4) \oplus (5) 
        \oplus \ZZ_3[a_{10},x_{26}] 
                \otimes \Delta (a_8)\{a_9b_{12}b_{18}^2\} 
        \oplus (7)$. 

\mathvskip\noindent
Then one can verify that 
the following elements are chosen to be generators 
of the filtered algebra $E_0$:
$$\begin{array}{llll}
      \overline{y}_{20} = -a_4b_{16}, 
    & \overline{y}_{22} = a_4b_{18}, 
    & \overline{y}_{26} = a_8b_{18},\\
      \overline{y}_{58} = -a_4^2b_{16}^2b_{18},\quad
    & \overline{y}_{60} = -a_4^2b_{16}b_{18}^2,\quad
    & \overline{y}_{64} = -a_8^2b_{12}b_{18}^2, \quad
    & \overline{y}_{76} = -a_4^2b_{16}^2b_{18}^2,\\
      \overline{y}_{21} = a_9b_{12},
    & \overline{y}_{25} = a_9b_{16},
    & \overline{y}_{27} = a_9b_{18}.
\end{array}$$ 

Numbering the term 
\ $\ZZ_3[a_{10},x_{26}]\{a_9b_{12}b_{18}, a_9b_{16}b_{18},
                a_9b_{16}b_{18}^2, a_9b_{18}^2\}$ %
\ by (3) and the term
\ $\ZZ_3[a_{10},x_{26}] \otimes \Delta (a_8)\{a_9b_{12}b_{18}^2\}$ %
\ by (6), 
we have proved the following

\begin{Lemma}\label{36}
{\it We have that }

\mathvskip
\mbox{}\hskip3em 
$E_4 = \Big( \ZZ_3[a_4,a_8,a_{10},x_{26}]
          \{1,y_{20},y_{20}^2,y_{22},y_{22}^2,y_{20}y_{22},
                                        y_{58},y_{60},y_{76}\}$

$\mbox{}\hskip7em 
        \oplus \ZZ_3[a_8,a_{10},x_{26}]
           \{y_{26},y_{26}^2,y_{20}y_{26},y_{22}y_{26},y_{64}\}$ 
                                                \mbox{}\quad\hfill $(1)$

$\mbox{}\hskip7em 
        \oplus \ZZ_3[x_{26}]
           \{a_9,y_{21},y_{25},y_{27},
                y_{21}a_8,y_{21}a_{10},y_{25}a_{10},y_{21}y_{26}\} $
                                                \mbox{}\quad\hfill $(2)$

$\mbox{}\hskip7em 
        \oplus \ZZ_3[a_{10},x_{26}]
           \{a_9b_{16}b_{18}, a_9b_{16}b_{18}^2, 
                a_9b_{12}b_{18}, a_9b_{18}^2\}$
                                                \mbox{}\quad\hfill $(3)$

$\mbox{}\hskip7em 
        \oplus \ZZ_3[a_8,a_{10},x_{26}]
           \{a_9b_{16}^2, a_9b_{16}^2b_{18}, a_9b_{16}^2b_{18}^2\}$
                                                \mbox{}\quad\hfill $(4)$

$\mbox{}\hskip7em 
        \oplus \ZZ_3[a_4,a_8,a_{10},x_{26}]
           \{a_9b_{12}^2, a_9b_{12}^2b_{18}, a_9b_{12}^2b_{18}^2,
                a_9b_{12}^2b_{16}, a_9b_{12}^2b_{16}b_{18},$

$\mbox{}\hskip17em 
             a_9b_{12}^2b_{16} b_{18}^2, a_9b_{12}^2b_{16}^2,
             a_9b_{12}^2b_{16}^2b_{18}, a_9b_{12}^2b_{16}^2b_{18}^2\}$
                                                \hspace{6pt}$(5)$

$\mbox{}\hskip7em 
        \oplus \ZZ_3[a_{10},x_{26}] \otimes \Delta(a_8)
          \{a_9b_{12} b_{18}^2\}$               \mbox\quad{}\hfill $(6)$

$\mbox{}\hskip7em 
        \oplus \ZZ_3[a_8,a_{10},x_{26}]
           \{a_9b_{12}b_{16}, a_9b_{12}b_{16}b_{18},
                a_9b_{12}b_{16}b_{18}^2, $

$\mbox{}\hskip20em 
           a_9b_{12}b_{16}^2, a_9b_{12}b_{16}^2b_{18},
           a_9b_{12}b_{16}^2b_{18}^2\} \,\Big)$
                                                \mbox{}\quad\hfill $(7)$

$\mbox{}\hskip6em \otimes \ZZ_3[x_{36},x_{48},x_{54}],$

\mathvskip\noindent
{\it where the cocycles} $y_i$ {\it and} $x_i$ {\it are defined 
as follows}:
$$\begin{array}{llll}
      y_{20} = a_8b_{12} - a_4b_{16},\quad
    & y_{22} = a_4b_{18} - a_{10}b_{12},\quad
    & y_{26} = a_8b_{18} - a_{10}b_{16},\quad &
\\ 
      y_{58} = \partial^2(b_{12}^2b_{16}^2b_{18}),
    & y_{60} = \partial^2(b_{12}^2b_{16}b_{18}^2),
    & y_{64} = \partial^2(b_{12}b_{16}^2b_{18}^2),
    & y_{76} = \partial^2(b_{12}^2b_{16}^2b_{18}^2), 
\\ 
      y_{21} = a_9b_{12} - c_{17}a_4,
    &  y_{25} = a_9b_{16} - c_{17}a_8,
    & y_{27} = a_9b_{18} - c_{17}a_{10},
\\
      x_{36} = b_{12}^3, \quad 
    & x_{48} = b_{16}^3, \quad 
    &x_{54} = b_{18}^3.
\end{array}$$
\end{Lemma}

As for the auxiliary derivation $\partial$ defined in (\ref{eq:33}) 
we have the following formula:

\begin{Lemma}\label{37}
{\it Let $Q$ be an element of %
\ $\ZZ_3[a_4,a_8, a_{10}, b_{12}, b_{16}, b_{18}]$. Then we have} 
\begin{equation}
x_{26} \partial^2(-Q) = d(a_9Q + c_{17} \partial Q). 
\label{eq:38}
\end{equation}
\end{Lemma}

\begin{proof}
Using a formula %
\ $b_i^n a_9 = a_9b_i^n + n c_{17} a_{i-8}b_i^{n-2}$, 
one can show by a tedious computation that
$$d(a_9Q) = -a_9^2 \partial Q - a_9 c_{17}\partial^2Q\ \quad \mbox{and}
\quad d(c_{17}\partial Q) = a_9^2\partial Q + c_{17}a_9\partial^2Q,$$
from which the lemma follows. 
\end{proof}

%
Looking at the expression of Lemma \ref{36}, 
we see that 
all the possible non permanent cocycles appear 
in the terms from (3) to (7).
Consider an element 
$$-x_{26}\partial^2 Q = d(a_9Q + c_{17}\partial Q)$$ 
in (\ref{eq:38}), 
where $Q = f(a_4, a_8, a_{10})b_{12}^i b_{16}^j b_{18}^k
\in \ZZ_3[a_4, a_8, a_{10}, b_{12}, b_{16}, b_{18}]$. Then we have
$$ w(a_9Q) = w(a_9) + w(Q) = 1 + w(f).$$
We also have
$$ w(\partial   Q) \ge 2 + w(f) \quad \hbox{and} \quad 
   w(\partial^2 Q) \ge 4 + w(f),$$
since $\partial Q$ and $\partial^2 Q$ are of the form 
$$\Sum a_{\mu} f(a_4, a_8, a_{10})b_{12}^k b_{16}^{\ell} b_{18}^m
\quad \hbox{and} \quad
\Sum a_{\mu}a_{\nu} f(a_4, a_8, a_{10})b_{12}^p b_{16}^q b_{18}^r$$ 
respectively. Thus we have
$$\begin{array}{l}
w(a_9Q + c_{17}\partial Q) = 1 + w(f); \\
w(-x_{26}\partial^2 Q) = w(x_{26})+ w(\partial^2 Q) 
\ge 3 + 4 + w(f) = 6 + w(a_9Q + c_{17}\partial Q).
\end{array}$$
So we have 
$$w(-x_{26}\partial^2 Q) - w(a_9Q + c_{17}\partial Q) \ge 6.$$
Therefore by the formula (\ref{eq:38}) 
we see that 
the differential $d$ increases the filtration degree 
defined in the paragraph below (\ref{eq:32}) 
by more than 5.
This means that $d_4=0$ and $d_5=0$, whence $E_4 = E_5 = E_6$.

Now it follows from Lemma \ref{37} 
that 
$$d_6(a_9Q) = -x_{26}\partial^2Q.$$
\indent
To calculate the $E_7$-term, 
it is more convenient to regard the $\wtilde{E}_7$-term (or 
equivalently the $E_4 = E_6$-term) expressed 
in Lemma \ref{36} 
as being equipped with the differential $d_6$. 
Then the filtration of $E_6$ defined similarly by (\ref{eq:34}) 

\bigskip\noindent
$(3.4)' \qquad
\begin{array}[t]{l}
\wtilde{w}'(b_i)=0 \quad \hbox{for} \quad i=12,16,18,\\
\wtilde{w}'(x_{26})=\wtilde{w}'(a_9)=\wtilde{w}'(a_4)=0, \quad
\wtilde{w}'(a_8)=1,\quad
\wtilde{w}'(a_{10})=3
\end{array}$

\bigskip\noindent
induces a spectral sequence %
$\{\wtilde{E}_r', \wtilde{d}_r'\}$ such that %
$\wtilde{E}_0' = E_6$ and $\wtilde{E}_\infty' = E_7$. 
Then we see that the differential $\wtilde{d}_0'$ is non trivial 
only on (5) which is mapped to (1), 
where we have by the above formula and (\ref{eq:33}) 
 that 
$$\begin{array}{lll}
      \wtilde{d}_0'(a_9b_{12}^2)               = -x_{26}a_4^2,
   \ & \wtilde{d}_0'(a_9b_{12}^2b_{18})         = -x_{26}a_4 y_{22},
   \ & \wtilde{d}_0'(a_9b_{12}^2b_{18}^2)       = -x_{26}y_{22}^2,\\
      \wtilde{d}_0'(a_9b_{12}^2b_{16})         = -x_{26}a_4 y_{20},
   \ & \wtilde{d}_0'(a_9b_{12}^2b_{16}b_{18})   = -x_{26}y_{20}y_{22},
   \ & \wtilde{d}_0'(a_9b_{12}^2b_{16} b_{18}^2)= -x_{26}y_{60},\\
      \wtilde{d}_0'(a_9b_{12}^2b_{16}^2)       = -x_{26}y_{20}^2,
   \ & \wtilde{d}_0'(a_9b_{12}^2b_{16}^2b_{18}) = -x_{26}y_{58},
   \ & \wtilde{d}_0'(a_9b_{12}^2b_{16}^2b_{18}^2) = -x_{26}y_{76}.
\end{array}$$

\mathvskip
It follows from these relations that 

\mathvskip \mbox{}\hskip3em  
$\KER \wtilde{d}_0' = 0$, \nopagebreak

\mbox{}\hskip3em  
$\Coker \wtilde{d}_0' 
   = \ZZ_3[a_4,a_8,a_{10},x_{26}]/(x_{26}a_4^2) 
        \oplus \Big( \ZZ_3[a_4,a_8,a_{10},x_{26}]/(x_{26}a_4) \Big)
                                                 \{y_{20},y_{22}\}
\\ \mbox{}\hskip17em 
        \oplus \ZZ_3[a_4,a_8,a_{10}]
            \{y_{20}^2,y_{22}^2,y_{20}y_{22},y_{58},y_{60},y_{76}\}$.

Thus we obtain that 
$$\wtilde{E}_1' 
  = \Coker \wtilde{d}_0' 
        \oplus \ZZ_3[a_8,a_{10},x_{26}]
           \{y_{26},y_{26}^2,y_{20}y_{26},y_{22}y_{26},y_{64}\}
        \oplus (2) \oplus (3) \oplus (4) \oplus (6) \oplus (7).$$
Now the differential $\wtilde{d}_1'$ is non trivial only on (7) 
which is mapped to the following 
$$\Coker \wtilde{d}_0' \oplus \ZZ_3[a_8,a_{10},x_{26}]\{y_{26},
y_{26}^2,y_{20}y_{26},y_{22}y_{26},y_{64}\}.$$ 
Here to determine $\wtilde{d}_1'(a_9b_{12}b_{16}b_{18})$, 
we use the formula 
$$a_4 y_{26} +a_8 y_{22} 
        = -a_8 y_{22} +a_{10}y_{20} 
        = -a_{10}y_{20} - a_4 y_{26} 
        = \partial^2(b_{12}b_{16}b_{18}),
$$
from which we see that %
$-a_8 y_{22} +a_{10}y_{20} = -a_8 y_{22}$  %
on $\wtilde{E}_1'$, since $-a_8 y_{22}$ and $a_{10}y_{20}$ are
of filtrations 1 and 2 respectively, and hence we have %
$\wtilde{d}_1'(a_9b_{12}b_{16}b_{18}) = a_8 x_{26}y_{22}$.
For the other elements of (7), the calculation of $\wtilde{d}_1'$ is 
straightforward from (\ref{eq:38}), 
and the results are as follows:
$$\begin{array}{lll}
     \wtilde{d}_1'(a_9b_{12}b_{16}) = -x_{26}a_4a_8,
   & \wtilde{d}_1'(a_9b_{12}b_{16}b_{18}^2) = -x_{26}y_{22}y_{26},\\
     \wtilde{d}_1'(a_9b_{12}b_{16}^2) = -x_{26}a_8 y_{20},
 \ & \wtilde{d}_1'(a_9b_{12}b_{16}^2b_{18}) = -x_{26}y_{20}y_{26},
 \ & \wtilde{d}_2'(a_9b_{12}b_{16}^2b_{18}^2) = -x_{26}y_{64}.
\end{array}$$
Hence we have that 

\mathvskip \mbox{}\hskip3em  
$\KER \wtilde{d}_1' = 0$,

\mbox{}\hskip3em  
$\Coker \wtilde{d}_1' 
        = \ZZ_3[a_4,a_8,a_{10},x_{26}]/(x_{26}a_4^2, x_{26}a_4a_8)$

$\mbox{}\hskip10em 
        \oplus \Big(\ZZ_3[a_4,a_8,a_{10},x_{26}]
                /(x_{26}a_4, x_{26}a_8)\Big)\{y_{20},y_{22}\}$

$\mbox{}\hskip10em 
        \oplus \ZZ_3[a_4,a_8,a_{10}]
           \{y_{20}^2,y_{22}^2,y_{20}y_{22},y_{58},y_{60},y_{76}\}$
                                        \mbox{}\quad\hfill $\b{1}''$

$\mbox{}\hskip10em 
        \oplus \ZZ_3[a_8,a_{10},x_{26}]\{y_{26},y_{26}^2,y_{64}\} 
        \oplus \ZZ_3[a_8,a_{10}]\{y_{20}y_{26},y_{22}y_{26}\}$.

\mathvskip\noindent
Thus we obtain that 
$$\wtilde{E}_2' 
        = \Coker \wtilde{d}_1' 
                \oplus (2) \oplus (3) \oplus (4) \oplus (6).$$
For degree reasons we see that the elements %
$a_9b_{16}b_{18}$, $a_9b_{16}b_{18}^2$, $a_9b_{18}^2$ in (3) are %
$\wtilde{d}_2'$-cycles and so they survive to $\wtilde{E}_3'$.

Observe that the differential $\wtilde{d}_2'$ is non trivial
only on (4).
The differential $\wtilde{d}_2'$ maps (4) to the following summands
$$
\ZZ_3[a_4,a_8,a_{10},x_{26}]/(x_{26}a_4^2, x_{26}a_4a_8) 
        \oplus \ZZ_3[a_8,a_{10},x_{26}]\{y_{26},y_{26}^2\}.
$$
It follows from (\ref{eq:38}) 
that 

\mathvskip \mbox{}\hskip3em 
$\wtilde{d}_2'(a_9b_{16}^2)         = -x_{26}a_8^2, \quad
        \wtilde{d}_2'(a_9b_{16}^2b_{18})   = -x_{26}a_8y_{26}, \quad
        \wtilde{d}_2'(a_9b_{16}^2b_{18}^2) = -x_{26}y_{26}^2$,

\mbox{}\hskip3em 
$\wtilde{d}_2'(a_9b_{12}b_{16}^2b_{18}^2) = -x_{26}y_{64}$.

\mathvskip \noindent
First we calculate the differential 
$$\wtilde{d}_2' : 
        \ZZ_3[a_8,a_{10},x_{26}]\{a_9b_{16}^2\}
        \to
        \ZZ_3[a_4, a_8, a_{10},x_{26}]/(x_{26}a_4^2, x_{26}a_4a_8),$$
where we have that 
$$\ZZ_3[a_4, a_8, a_{10},x_{26}]/x_{26}(a_4^2, a_4a_8)
\cong 
        \ZZ_3[a_4, a_8, a_{10}]  \oplus 
        \ZZ_3[a_4, a_8]/(a_4^2, a_4a_8)  \otimes  
        \ZZ_3[a_{10}]  \otimes  
        \ZZ_3[x_{26}]^+.$$
We see  that 

\mathvskip \mbox{}\hskip3em 
$\KER \wtilde{d}_2' \mid
        \ZZ_3[a_8,a_{10},x_{26}]\{a_9b_{16}^2\} = 0$,

\mbox{}\hskip3em 
$\Coker \wtilde{d}_2' 
        \cong \ZZ_3[a_4,a_8,a_{10},x_{26}]
                /x_{26}(a_4^2, a_4a_8, a_8^2)$
\begin{equation}
\label{eq:39}
\hskip3.8em \cong
        \ZZ_3[a_4,a_8,a_{10}] \oplus 
        \ZZ_3[a_4,a_8]/(a_4^2, a_4a_8, a_8^2)
        \otimes \ZZ_3[a_{10}] \otimes \ZZ_3[x_{26}]^+.
\end{equation}

\mathvskip \noindent
Also we easily see that

\mathvskip \mbox{}\hskip3em 
$\KER \wtilde{d}_2' \mid
        \ZZ_3[a_8, a_{10},x_{26}]\{a_9b_{16}^2b_{18},a_9b_{12}b_{16}^2b_{18}^2\} \oplus
        \ZZ_3[a_{10},x_{26}]\{a_9b_{16}^2b_{18}^2 \}=0$, 

\mbox{}\hskip3em 
$\Coker \wtilde{d}_2' 
        = \Big(\ZZ_3[a_8,a_{10},x_{26}]/x_{26}a_8 \Big)\{y_{26} \}
                \oplus \ZZ_3[a_8, a_{10}]\{y_{26}^2,y_{64} \}$.

\mathvskip \noindent
When we set

\mathvskip \mbox{}\hskip3em 
$
\begin{array}{ll}
      F = & \Big( \ZZ_3[a_4,a_8,a_{10},x_{26}]
                /x_{26}(a_4^2, a_4a_8, a_8^2)\\
        &\mbox{}\oplus \Big(\ZZ_3[a_4,a_8,a_{10},x_{26}]
                /(x_{26}a_4, x_{26}a_8) \Big) \{y_{20}, y_{22} \}\\
        &\mbox{}\oplus \ZZ_3[a_4,a_8,a_{10}]
         \{y_{20}^2,y_{22}^2,y_{20}y_{22},y_{58},y_{60},y_{76} \}\\
        &\mbox{}\oplus \Big(\ZZ_3[a_8,a_{10},x_{26}]/x_{26}a_8 \Big)
                                                         \{y_{26}\}\\
        &\mbox{}\oplus \ZZ_3[a_8, a_{10}]
                \{y_{20}y_{26},y_{22}y_{26},y_{26}^2,y_{64}\} 
             \Big) \otimes \ZZ_3[x_{36},x_{48},x_{54}],
\end{array}$

\mathvskip \noindent
we obtain
$$\wtilde{E}_3' 
        = F \oplus (2) \oplus (3) \oplus (6).$$
We observe that the differential $\wtilde{d}_3'$ is non trivial only
on the summand
\[\ZZ_3[a_{10},x_{26}]\{a_9b_{12}b_{18}\}
        \oplus \ZZ_3[a_{10},x_{26}]
        \otimes \Delta(a_8)\{a_9b_{12}b_{18}^2\}.\]

It follows from (\ref{eq:38}) 
that
$$
        \wtilde{d}_3'(a_9b_{12}b_{18})   = -x_{26}a_4a_{10}, \quad
        \wtilde{d}_3'(a_9b_{12}b_{18}^2) = -x_{26}a_{10}y_{22}.
$$
First we calculate
$$\wtilde{d}_3' : 
        \ZZ_3[a_{10},x_{26}]\{a_9b_{12}b_{18}\}
        \to
        \ZZ_3[a_4, a_8, a_{10},x_{26}]/x_{26}(a_4^2,a_4a_8,a_8^2).$$
Using (\ref{eq:39}), 
we obtain that
\mathvskip \mbox{}\hskip3em 
$\KER \wtilde{d}_3' \mid \ZZ_3[a_{10},x_{26}]\{a_9b_{12}b_{18}\} = 0$,

\mbox{}\hskip-\arraycolsep 
$\begin{array}{ll}
\mbox{}\hskip3em
\Coker \wtilde{d}_3' 
        & \cong \ZZ_3[a_4,a_8,a_{10},x_{26}]
                /x_{26}(a_4^2, a_4a_8, a_4a_{10}, a_8^2)
\end{array}$
\begin{equation}
\hskip3.6em 
\cong
        \ZZ_3[a_4,a_8,a_{10}] \oplus 
        \ZZ_3[a_4,a_8,a_{10}]/(a_4^2, a_4a_8, a_4a_{10}, a_8^2)
        \otimes
        \ZZ_3[x_{26}]^+.
\label{eq:310}
\end{equation}

\mathvskip \noindent
By a similar way, we see that

\mathvskip \mbox{}\hskip3em 
$\KER \wtilde{d}_3' \mid
        \ZZ_3[a_{10},x_{26}] \otimes \Delta(a_8)\{a_9b_{12}b_{18}^2 \}
        = \ZZ_3[a_{10},x_{26}]\{a_9a_8b_{12}b_{18}^2 \}$, 

\mbox{}\hskip3em 
$\Coker \wtilde{d}_3' \mid
        \ZZ_3[a_{10},x_{26}] \otimes \Delta(a_8)\{a_9b_{12}b_{18}^2 \}
        = \ZZ_3[a_4,a_8,a_{10}]\{y_{22} \}
                \oplus \ZZ_3[x_{26}]^+ \{y_{22} \}$.

\mathvskip \noindent
Thus we obtain that 
$$\begin{array}{ll}
\wtilde{E}_4' = & 
\Big(
  \ZZ_3[a_4,a_8,a_{10},x_{26}]
                /x_{26}(a_4^2, a_4a_8, a_4a_{10}, a_8^2) 
                                                \\ & \mbox{}\quad
        \oplus \Big(\ZZ_3[a_8,a_{10},x_{26}]/(x_{26}a_8)\Big)\{y_{26}\}
        \oplus \ZZ_3[a_8,a_{10}]\{y_{26}^2\}
                                                \\ & \mbox{}\quad
        \oplus \Big( \ZZ_3[a_4,a_8,a_{10},x_{26}]/x_{26}(a_4,a_8) \Big)
                                                 \{y_{20}\} 
                                                \\ & \mbox{}\quad
        \oplus \ZZ_3[a_4,a_8,a_{10}]\{y_{22}\}
        \oplus \ZZ_3[x_{26}]^+\{y_{22}\}  
                                                \\ & \mbox{}\quad
        \oplus \ZZ_3[a_4,a_8,a_{10}]
            \{y_{20}^2,y_{22}^2,y_{20}y_{22},y_{58},y_{60},y_{76}\}
                                                \\ & \mbox{}\quad
        \oplus \ZZ_3[a_8,a_{10}]
           \{y_{20}y_{26},y_{22}y_{26},y_{64}\}
                                                \\ & \mbox{}\quad
        \oplus \ZZ_3[x_{26}]\{a_9,y_{21},y_{25},y_{27}, 
           y_{21}a_8,y_{21}a_{10},y_{25}a_{10},y_{21}y_{26}\}
                                                \\ & \mbox{}\quad
        \oplus \ZZ_3[a_{10},x_{26}]\{a_9b_{16}b_{18}, 
           a_9b_{16}b_{18}^2, a_9b_{18}^2\} 
                                                \\ & \mbox{}\quad
        \oplus \ZZ_3[a_{10},x_{26}]\{a_9a_8b_{12}b_{18}^2\} 
\Big)                                           \\ & \mbox{}\quad
\mbox{}\hskip-1em
\otimes \ZZ_3[x_{36},x_{48},x_{54}].
\end{array}$$
\indent
The following are the only non trivial differentials on $\wtilde{E}_4'$:
$$\begin{array}{ll}
\wtilde{d}_4':
\ZZ_3[a_{10},x_{26}]\{a_9b_{16}b_{18}\} 
        &\rightarrow \hphantom{\Big(}\ZZ_3[a_4,a_8,a_{10},x_{26}]
                        / x_{26}(a_4^2, a_4a_8, a_4a_{10}, a_8^2), \\
\wtilde{d}_4':
\ZZ_3[a_{10},x_{26}]\{a_9b_{16}b_{18}^2\} 
        &\rightarrow \Big(\ZZ_3[a_8,a_{10},x_{26}]
                        / x_{26}a_8\Big)\{y_{26}\}, \\
\end{array}$$
which are given by
$$ \wtilde{d}_4'(a_9b_{16}b_{18}) = -x_{26}a_8a_{10}, 
        \quad 
   \wtilde{d}_4'(a_9b_{16}b_{18}^2) = -x_{26}a_{10}y_{26}$$
respectively. Hence we have respectively that

\mathvskip\mbox{}\hskip3em 
$\KER \wtilde{d}_4'\mid \ZZ_3[a_{10},x_{26}]\{a_9b_{16}b_{18}\}= 0$,

\mbox{}\hskip3em 
$\Coker \wtilde{d}_4'\mid \ZZ_3[a_{10},x_{26}]\{a_9b_{16}b_{18}\} \\
\mbox{}\hskip8.4em 
        = \ZZ_3[a_4,a_8,a_{10},x_{26}]
                / x_{26}(a_4^2, a_4a_8, a_4a_{10}, a_8^2, a_8a_{10}) \\
\mbox{}\hskip8.4em 
        = \ZZ_3[a_4,a_8,a_{10}] \oplus 
          \ZZ_3[a_4,a_8,a_{10}] 
                / (a_4^2, a_4a_8, a_4a_{10}, a_8^2, a_8a_{10}) 
        \otimes \ZZ_3[x_{26}]^+$ 
\begin{equation}
 \label{eq:311} \hskip4.44em      = \ZZ_3[a_4,a_8,a_{10}] \oplus \ZZ_3\{a_4,a_8\} 
        \otimes \ZZ_3[x_{26}]^+ 
          \oplus \ZZ_3[a_{10}]\otimes \ZZ_3[x_{26}]^+;
\end{equation}

\medskip
\mbox{}\hskip3em 
$\KER \wtilde{d}_4'\mid \ZZ_3[a_{10},x_{26}]\{a_9b_{16}b_{18}^2\}= 0$,

\mbox{}\hskip3em\hskip-1\arraycolsep 
$\begin{array}{ll}
\Coker \wtilde{d}_4'\mid \ZZ_3[a_{10},x_{26}]\{a_9b_{16}b_{18}^2\} 
        &= \Big(\ZZ_3[a_8,a_{10},x_{26}]/x_{26}(a_8,a_{10})\Big)
                                                        \{y_{26}\}\\
        &= (\ZZ_3[a_8,a_{10}] \oplus \ZZ_3[x_{26}]^+)\{y_{26}\}.
\end{array}$

\mathvskip
If we set a submodule
\begin{equation}
\begin{array}{ll}
L = & 
        \ZZ_3[a_8,a_{10}]\{ y_{26} \}
        \oplus \ZZ_3[x_{26}]^+\{y_{26}\} 
                                                \\ & \mbox{}\quad
        \oplus \ZZ_3[a_4,a_8,a_{10}]\{y_{22}\}
        \oplus \ZZ_3[x_{26}]^+\{y_{22}\}
                                                \\ & \mbox{}\quad
        \oplus \ZZ_3[a_4,a_8,a_{10}]
            \{y_{20}^2,y_{22}^2,y_{20}y_{22},y_{58},y_{60},y_{76}\}
                                                \\ & \mbox{}\quad
        \oplus \ZZ_3[a_8,a_{10}]\{y_{20}y_{26},y_{22}y_{26},y_{64}\}
                                                \\ & \mbox{}\quad
        \oplus \ZZ_3[x_{26}]\{a_9,y_{21},y_{25},y_{27}, 
           y_{21}a_8,y_{21}a_{10},y_{25}a_{10},y_{21}y_{26}\},
\end{array}\label{eq:312}
\end{equation}
we obtain that 
$$\begin{array}{ll}
\wtilde{E}_5' =  
\Big(L \hskip-\arraycolsep & \mbox{}\oplus 
  \ZZ_3[a_4,a_8,a_{10},x_{26}]
                /x_{26}(a_4^2, a_4a_8, a_4a_{10}, a_8^2, a_8a_{10}) 
                                                \\ & \mbox{}
        \oplus \ZZ_3[a_4,a_8,a_{10},x_{26}]/x_{26}(a_4,a_8) \{y_{20}\} 
                                                \\ & \mbox{}
        \oplus \ZZ_3[a_{10},x_{26}]\{a_9a_8b_{12}b_{18}^2\} 
        \oplus \ZZ_3[a_{10},x_{26}]\{a_9b_{18}^2\} 

\Big)
\otimes \ZZ_3[x_{36},x_{48},x_{54}].
\end{array}$$

\mathvskip
To determine $\wtilde{E}_6'$, we need the following

\begin{Lemma}
$\wtilde{d}_5'(a_9a_8b_{12}b_{18}^2) = - x_{26}a_{10}^2y_{20}$.
\end{Lemma}
\begin{proof}
It follows from (\ref{eq:38}) 
that 
$$d\Big(a_8(a_9b_{12}b_{18}^2 - c_{17}\partial(b_{12}b_{18}^2))\Big) 
        = -x_{26}a_8 a_{10}y_{22}.$$
Using the relation 
\ $-a_8 y_{22} +a_{10}y_{20} = \partial(b_{12}b_{16}b_{18})$ %
\ and the formula (\ref{eq:38}), 
we have that 
$$d\Big(a_{10}(a_9b_{12}b_{16}b_{18} 
                - c_{17}\partial(b_{12}b_{16}b_{18}))\Big) 
        = -a_{10}x_{26}(-a_8 y_{22} +a_{10}y_{20}).$$
Hence we obtain that 
$$d\Big(a_9a_8b_{12}b_{18}^2 
        - c_{17}a_8\partial(b_{12}b_{18}^2) 
         +a_9a_{10}b_{12}b_{16}b_{18} 
         + c_{17}a_{10}\partial(b_{12}b_{16}b_{18})\Big) 
     = -x_{26}a_{10}^2 y_{20},$$
which implies the lemma. 
\end{proof}

Therefore the only non trivial differential
$$\wtilde{d}_5':
\ZZ_3[a_{10},x_{26}]\{a_9a_8b_{12}b_{18}^2\} 
        \rightarrow \Big(\ZZ_3[a_4,a_8,a_{10},x_{26}]
                        / x_{26}(a_4, a_8)\Big)\{y_{20}\}$$
is given by $
   \wtilde{d}_5'(a_9a_8b_{12}b_{18}^2) = -x_{26}a_{10}^2y_{20}$.

Hence we have

\mathvskip \mbox{}\hskip3em 
$\KER \wtilde{d}_5'\mid \ZZ_3[a_{10},x_{26}]\{a_9a_8b_{12}b_{18}^2\}
                                = 0$,

\mbox{}\hskip3em\hskip-\arraycolsep 
$\begin{array}{ll}
\Coker \wtilde{d}_5'\mid \ZZ_3[a_{10},x_{26}]\{a_9a_8b_{12}b_{18}^2\} 
        &= \Big(\ZZ_3[a_4,a_8,a_{10},x_{26}]/x_{26}(a_4,a_8,a_{10}^2)\Big)
                                                        \{y_{20}\}\\
        &= \ZZ_3[a_4,a_8,a_{10}]\{y_{20}\} 
        \oplus \ZZ_3[x_{26}]^+\{y_{20}, a_{10}y_{20}\}.
\end{array}$

\mathvskip
Thus we obtain that 
$$\begin{array}{ll}
\wtilde{E}_6' =  
\Big(L \hskip-\arraycolsep & \mbox{} \oplus 
  \ZZ_3[a_4,a_8,a_{10},x_{26}]
                /x_{26}(a_4^2, a_4a_8, a_4a_{10}, a_8^2, a_8a_{10}) 
                                                \\ & \mbox{}
        \oplus \ZZ_3[a_4,a_8,a_{10}]\{y_{20}\}
        \oplus \ZZ_3[x_{26}]^+\{y_{20}, a_{10}y_{20}\}\\ 
        & \mbox{}\hskip12em
        \oplus \ZZ_3[a_{10},x_{26}]\{a_9b_{18}^2\} 
\Big) 
\otimes \ZZ_3[x_{36},x_{48},x_{54}],
\end{array}$$
where we recall that $L$ is defined at (\ref{eq:312}).

Finally we observe that the following is 
the only non trivial differential
$$\wtilde{d}_6':\ZZ_3[a_{10},x_{26}]\{a_9b_{18}^2\} 
        \rightarrow \ZZ_3[a_4,a_8,a_{10},x_{26}]
                /x_{26}(a_4^2, a_4a_8, a_4a_{10}, a_8^2, a_8 a_{10})$$
given by \ $\wtilde{d}_6'(a_9b_{18}^2) = -x_{26}a_{10}^2$.  
Using (\ref{eq:311}), 
we have the presentation

\mathvskip\mbox{}\hskip3em 
$\KER \wtilde{d}_6' \mid \ZZ_3[a_{10},x_{26}]\{a_9b_{18}^2\} = 0$,

\mbox{}\hskip3em 
$\Coker \wtilde{d}_6' \mid \ZZ_3[a_{10},x_{26}]\{a_9b_{18}^2\} 
        = \ZZ_3[a_4,a_8,a_{10}] 
        \oplus \ZZ_3[x_{26}]^+\{1, a_4, a_8,a_{10}\}$. 

\mathvskip\noindent
Thus we obtain that 
$$\begin{array}{rl}
\wtilde{E}_7' = & 
  \Big(
  \ZZ_3[a_4,a_8,a_{10}]\{1,y_{20},y_{20}^2,y_{22},
                y_{22}^2,y_{20}y_{22},y_{58},y_{60},y_{76}\}
                                                \\ & \mbox{}\quad
        \oplus \ZZ_3[a_8,a_{10}]
                \{y_{26},y_{26}^2,y_{20}y_{26},y_{22}y_{26},y_{64}\}
                                                \\ & \mbox{}\quad
        \oplus \ZZ_3[x_{26}]^+
                \{1, a_4, a_8,a_{10},y_{20},a_{10}y_{20},y_{22},y_{26}\}
                                                \\ & \mbox{}\quad
        \oplus \ZZ_3[x_{26}]\{a_9,y_{21},y_{25},y_{27},
                y_{21}a_8,y_{21}a_{10},y_{25}a_{10},y_{21}y_{26}\} 
\Big)                                           \\ & \mbox{}\quad
\mbox{}\hskip-1em
\otimes \ZZ_3[x_{36},x_{48},x_{54}].
\end{array}$$

Since all the generators $a_i$, $x_j$, $y_k$ chosen in Lemma \ref{36} 
can be seen to be cocycles by direct calculation 
in the differential algebra $(\overline{V},d)$, 
we see that \ $\wtilde{E}_7' = \wtilde{E}_\infty' = E_7$ %
\ and \ $E_7 = E_\infty$. 
Thus we have proved Theorem \ref{21}.

\section{Proof of Theorem 2.3
}

Using the presentation of Theorem \ref{21}, 
we can list up all the relations, which are calculated 
in the complex $(\overline{V},d)$ defined in Section 3.
The calculation is usually straightforward, although it becomes more
systematic when using the auxiliary derivation $\partial$ defined
in (\ref{eq:33}).

First we need two lemmas.
\begin{Lemma}\label{41}
We have that 
\[a_9\partial^2Q = y_{21}\partial^2Q = y_{25}\partial^2Q
                 = y_{27}\partial^2Q = x_{26}\partial^2Q = 0,\]
where $Q$ is an element of 
$\ZZ_3[a_4,a_8,a_{10},b_{12},b_{16},b_{18}]$.
\end{Lemma}
\begin{proof}
$\begin{array}[t]{ll}
a_9\partial^2Q = d(\partial Q), 
&y_{21}\partial^2Q = d(a_4Q + b_{12}\partial Q),\\
y_{25}\partial^2Q = d(a_8Q + b_{16}\partial Q),\qquad 
&y_{27}\partial^2Q = d(a_{10}Q + b_{18}\partial Q).
\end{array}$\\ 
The last equation was proved in Lemma \ref{37}.
\end{proof}

\begin{Lemma}\label{42}
{\it The $\partial^2$-image is generated over}
\ $\ZZ_3[a_4,a_8,a_{10},x_{36}, x_{48},x_{54}]$
{\it by the elements in the table $(4.0)$ in the next page.}
\end{Lemma}

{\bf Proof of Theorem 2.3.B} \quad
\quad One can show the relations in \ i) by expressing 
the elements in terms of $a_i$ and $b_j$ 
and those in \ ii) by case by case checking as follows:
\renewcommand{\arraystretch}{0.95}
\[\begin{array}{llllll}
        \multicolumn{6}{l}{
      a_9^2=d(\!c_{17}\!), \qquad y_{21}^2=d(\!c_{17}b_{12}^2\!), \qquad 
      y_{25}^2=d(\!c_{17}b_{16}^2\!), \qquad y_{27}^2=d(\!c_{17}b_{18}^2\!),}\\
        a_9y_{21} + x_{26}a_4&=d(\!c_{17}b_{12}\!),
       & a_9y_{25} + x_{26}a_8&=d(\!c_{17}b_{16}\!),
       & a_9y_{27} + x_{26}a_{10}&=d(\!c_{17}b_{18}\!),\\
         y_{21}y_{25} + x_{26}y_{20}&=d(\!c_{17}b_{12}b_{16}\!),\ %
       & y_{21}y_{27} - x_{26}y_{22}&=d(\!c_{17}b_{12}b_{18}\!),\ %
       & y_{25}y_{27} - x_{26}y_{26}&=d(\!c_{17}b_{16}b_{18}\!).
\end{array}\]

The first five relations in \ iii) are immediate from Lemmas \ref{41} and \ref{37}. 
One can show the remaining relations by appealing 
to the relation \ $d(Q)=a_9\partial Q + c_{17}\partial^2 Q $ \ for %
\ $Q \in \ZZ_3[a_i,b_j]$ and the table in Lemma \ref{42}:
\[\begin{array}{l}
\begin{array}{lll}
 a_9a_4    = d(b_{12}), \hskip20mm 
&a_9a_8    = d(b_{16}), \hskip20mm 
&a_9a_{10} = d(b_{18}),\\
 y_{21}a_4    = d(b_{12}^2), \hskip20mm
&y_{25}a_8    = d(b_{16}^2), \hskip20mm
&y_{27}a_{10} = d(b_{18}^2),\\
\end{array}\\
\begin{array}{l}
 y_{21}a_8 +a_9y_{20} = y_{25}a_4 - a_9y_{20} = d(b_{12}b_{16}),\\
 y_{21}a_{10} - a_9y_{22} = y_{27}a_4 +a_9y_{22} = d(b_{12}b_{18}),\\
 y_{25}a_{10} - a_9y_{26} = y_{27}a_8 +a_9y_{26} = d(b_{16}b_{18}),\\
\end{array}\\
\begin{array}{ll}
 y_{21}y_{20} = d(-b_{12}^2b_{16}), \hskip30mm
 y_{25}y_{20} = d(-b_{12}b_{16}^2), \\
 y_{21}y_{22} = d(-b_{12}^2b_{18}), \hskip30mm 
 y_{27}y_{22} = d(-b_{12}b_{18}^2), \\
 y_{25}y_{26} = d(-b_{16}^2b_{18}), \hskip30mm 
 y_{27}y_{26} = d(-b_{16}b_{18}^2),\\
\multicolumn{2}{l}{
 y_{27}y_{20} - y_{25}y_{22} 
     = y_{25}y_{22} + y_{21}y_{26} 
     = -y_{27}y_{20} - y_{21}y_{26} = d(-b_{12}b_{16}b_{18}).
}
\end{array}
\end{array}\]
\renewcommand{\arraystretch}{1}

{\bf Proof of Theorem 2.3.A} 
\quad 
First we have to show that $D \otimes \ZZ_3[x_{36},x_{48},x_{54}]$ is 
an ideal of $\Cotor_{H^*(E_6)}^{} (\ZZ_3, \ZZ_3)$. 
We denote by $S$ a subalgebra
\[
\ZZ_3[a_4,a_8,a_{10},b_{12},b_{16},b_{18}]
\]
of $\overline{V}$ defined in Section 3. By Lemma \ref{31} 
we obtain 
a module isomorphism
\[\overline{V} \cong T(a_9,c_{17}) \otimes S 
        = T(a_9,c_{17})^+ \otimes S \oplus S.\]
Table (4.0).\\
\renewcommand{\arraystretch}{1.5}
\begin{footnotesize}
\mbox{}
\[\begin{array}{|l|l|ll|}
\quad Q & \quad \partial Q & \multicolumn{2}{l|}{\quad \partial^2 Q} \\ 
b_{12}                          &{-}a_4                         &\hpm 0     & \\ 
b_{16}                          &{-}a_8                         &\hpm 0     & \\ 
b_{18}                          &{-}a_{10}                      &\hpm 0     & \\ 
b_{12}^2                        &\hpm a_4b_{12}                 &{-}a_4^2     & \\ 
b_{16}^2                        &\hpm a_8b_{16}                 &{-}a_8^2     & \\ 
b_{18}^2                        &\hpm a_{10}b_{18}              &{-}a_{10}^2  & \\ 
b_{12}b_{16}                    &{-}a_4b_{16} {-}a_8b_{12}              &{-}a_4a_8    & \\ 
b_{12}b_{18}                    &{-}a_4b_{18} {-}a_{10}b_{12}   &{-}a_4a_{10} & \\ 
b_{16}b_{18}                    &{-}a_8b_{18} {-}a_{10}b_{16}   &{-}a_8a_{10} & \\ 
b_{12}^2b_{16}                  &\hpm a_4b_{12}b_{16} {-}a_8b_{12}^2
&               {-}a_4y_{20}  & \!\!=\! {-}a_4^2b_{16} {+}a_4a_8b_{12}\\ 
b_{12}b_{16}^2                  &{-}a_4b_{16}^2 {+}a_8b_{12}b_{16} 
&               \hpm a_8y_{20}  & \!\!=\! \hpm a_4a_8b_{16} {-}a_8^2b_{12} \\ 
b_{12}^2b_{18}                  &\hpm a_4b_{12}b_{18} {-}a_{10}b_{12}^2
&               {-}a_4y_{22}  & \!\!=\! {-}a_4^2b_{18} {+}a_4a_{10}b_{12}\\ 
b_{12}b_{18}^2                  &{-}a_4b_{18}^2 {+}a_{10}b_{12}b_{18} 
&               \hpm a_{10}y_{22}  & \!\!=\! \hpm a_4a_{10}b_{18} {-}a_{10}^2b_{12}\\ 
b_{16}^2b_{18}                  &\hpm a_8b_{16}b_{18} {-}a_{10}b_{16}^2
&               {-}a_8y_{26}  & \!\!=\! {-}a_8^2b_{18} {+}a_8a_{10}b_{16}\\ 
b_{16}b_{18}^2                  &{-}a_8b_{18}^2 {+}a_{10}b_{16}b_{18} 
&               \hpm a_{10}y_{26}  & \!\!=\! \hpm a_8a_{10}b_{18} {-}a_{10}^2b_{16}\\ 
b_{12}b_{16}b_{18}              &{-}a_4b_{16}b_{18} {-}a_8b_{12}b_{18} {-}a_{10}b_{12}b_{16}
&               \multicolumn{2}{l|}{{-}a_4y_{26} {+}a_{10}y_{20}
                        = {-}a_8y_{22} {-}a_{10}y_{20}
                        = a_4y_{26} {+}a_8y_{22}}\\ 
&&                       & \!\!=\! {-}a_4a_8b_{18} {-}a_4a_{10}b_{16} {-}a_8a_{10}b_{12}\\ 
b_{12}^2b_{16}^2                &\hpm a_4b_{12}b_{16}^2 {+}a_8b_{12}^2b_{16}
&               {-}y_{20}^2  & \!\!=\! {-}a_4^2b_{16}^2 {-}a_4a_8b_{12}b_{16} {-}a_8^2b_{12}^2\\ 
b_{12}^2b_{18}^2                &\hpm a_4b_{12}b_{18}^2 {+}a_{10}b_{12}^2b_{18}
&               {-}y_{22}^2  & \!\!=\! {-}a_4^2b_{18}^2 {-}a_4a_{10}b_{12}b_{18} {-}a_{10}^2b_{12}^2\\ 
b_{16}^2b_{18}^2                &\hpm a_8b_{16}b_{18}^2 {+}a_{10}b_{16}^2b_{18}
&               {-}y_{26}^2  & \!\!=\! {-}a_8^2b_{18}^2 {-}a_8a_{10}b_{16}b_{18} {-}a_{10}^2b_{16}^2\\ 
b_{12}^2b_{16}b_{18}            &\hpm a_4b_{12}b_{16}b_{18} {-}a_8b_{12}^2b_{18} {-}a_{10}b_{12}^2b_{16}
&               {-}y_{20}y_{22} & \!\!=\! {-}a_4^2b_{16}b_{18} {+}a_4a_8b_{12}b_{18} {+}a_4a_{10}b_{12}b_{16} {-}a_8a_{10}b_{12}^2 \\ 
b_{12}b_{16}^2b_{18}            & {-}a_4b_{16}^2b_{18} {+}a_8b_{12}b_{16}b_{18} {-}a_{10}b_{12}b_{16}^2
&               \hpm y_{20}y_{26}& \!\!=\! \hpm a_4a_8b_{16}b_{18} {-}a_4a_{10}b_{16}^2 {-}a_8^2b_{12}b_{18} {+}a_8a_{10}b_{12}b_{16} \\ 
b_{12}b_{16}b_{18}^2            & {-}a_4b_{16}b_{18}^2  {-}a_8b_{12}b_{18}^2{+}a_{10}b_{12}b_{16}b_{18}
&               {-}y_{22}y_{26} & \!\!=\! {-}a_4a_8b_{18}^2 {+}a_4a_{10}b_{16}b_{18} {+}a_8a_{10}b_{12}b_{18} {-}a_{10}^2b_{12}b_{16} \\ 
b_{12}^2b_{16}^2b_{18}
& \hpm a_4b_{12}b_{16}^2b_{18} 
    {+}a_8b_{12}^2b_{16}b_{18} 
    {-}a_{10}b_{12}^2b_{16}^2
& \hpm y_{58} & \!\!=\!
    {-} a_4^2b_{16}^2b_{18} 
    {-} a_4a_8b_{12}b_{16}b_{18}
    {+} a_4a_{10}b_{12}b_{16}^2
\\&&&\hfill
    {-} a_8^2b_{12}^2b_{18}
    {+} a_8a_{10}b_{12}^2b_{16}
\\ 
b_{12}^2b_{16}b_{18}^2          
& \hpm a_4b_{12}b_{16}b_{18}^2 
    {-}a_8b_{12}^2b_{18}^2 
    {+}a_{10}b_{12}^2b_{16}b_{18}
& \hpm y_{60}
& \!\!=\!
    {-}a_4^2b_{16}b_{18}^2 
    {+} a_4a_8b_{12}b_{18}^2
    {-} a_4a_{10}b_{12}b_{16}b_{18}
\\&&&\hfill
    {+} a_8a_{10}b_{12}^2b_{18}
    {-} a_{10}^2b_{12}^2b_{16}
\\ 
b_{12}b_{16}^2b_{18}^2          
& {-}a_4b_{16}^2b_{18}^2 
  {+}a_8b_{12}b_{16}b_{18}^2 
  {+}a_{10}b_{12}b_{16}^2b_{18}
& \hpm y_{64}
& \!\!=\! \hpm a_4a_8b_{16}b_{18}^2
  {+} a_4a_{10}b_{16}^2b_{18}
  {-} a_8^2b_{12}b_{18}^2 
\\&&&\hfill
  {-} a_8a_{10}b_{12}b_{16}b_{18}
  {-} a_{10}^2b_{12}b_{16}^2
\\ 
b_{12}^2b_{16}^2b_{18}^2        
& \hpm a_4b_{12}b_{16}^2b_{18}^2 {+}a_8b_{12}^2b_{16}b_{18}^2 {+}a_{10}b_{12}^2b_{16}^2b_{18}
&                       \hpm y_{76}
                         & \!\!=\!
                        {-}a_4^2b_{16}^2b_{18}^2 {-} a_4a_8b_{12}b_{16}b_{18}^2
                        {-} a_4a_{10}b_{12}b_{16}^2b_{18}
\\&&&\hfill
   {-} a_8^2b_{12}^2b_{18}^2
                        {-} a_8a_{10}b_{12}^2b_{16}b_{18}
                         {-} a_{10}^2b_{12}^2b_{16}^2
\\ 
\end{array}\]
\end{footnotesize}
 \renewcommand{\arraystretch}{1}
Now we will show that $T(a_9,c_{17})^+ \otimes S$ is a two sided ideal 
of $\overline{V}$. Recall the relations in $\overline{V}$ :
$$
\begin{array}{ll}
b_j a_9 = a_9b_j + c_{17}a_{j-8} \quad  &\mbox{for}\quad j = 12,16,18;\\
a_j a_9 = a_9a_j                        &\mbox{for}\quad j = 4,8,10;\\
Q c_{17} = c_{17} Q                     &\mbox{for}\quad Q \in S,
\end{array}$$
and also observe that
$$Q T(a_9,c_{17})^+ \otimes S \subset T(a_9,c_{17})^+ 
        \quad \mbox{for}\quad Q \in S$$
which is obtained by the repeated use of the above relations.
Thus $T(a_9,c_{17})^+ \otimes S$ is a two sided ideal of $\overline{V}$.
Let $Z(\overline{V})$ and $B(\overline{V})$ be the sub-groups
consisting of cocycles and coboundaries in $\overline{V}$ respectively.
Thus we have
$$
\begin{array}{l}
Z(\overline{V}) 
        = \Big( (T(a_9,c_{17})^+ \otimes S) \cap Z(\overline{V})\Big) 
                \oplus S \cap Z(\overline{V}),\\
B(\overline{V}) 
        = \Big( (T(a_9,c_{17})^+ \otimes S) \cap B(\overline{V})\Big) 
                \oplus S \cap B(\overline{V}).
\end{array}
$$
Using the fact that $T(a_9,c_{17})^+ \otimes S$ is 
a two sided ideal of $\overline{V}$, 
we can show that \[
(T(a_9,c_{17})^+ \otimes S) \cap Z(\overline{V}) %
~\mbox{and}~ (T(a_9,c_{17})^+ \otimes S) \cap Z(\overline{V})\] are both 
two sided ideals of $Z(\overline{V})$. Hence
$$\Frac{(T(a_9,c_{17})^+ \otimes S) \cap Z(\overline{V})}
       {(T(a_9,c_{17})^+ \otimes S) \cap B(\overline{V})}$$ 
is an ideal of $H^*(\overline{V}) = \Cotor_{H^*(E_6)}^{}(\ZZ_3, \ZZ_3)$.

Here we have

\begin{equation}
\begin{array}{ll}
& D \otimes \ZZ_3[x_{36},x_{48},x_{54}]\medskip
        = \Frac{(T(a_9,c_{17})^+ \otimes S) \cap Z(\overline{V})}
               {(T(a_9,c_{17})^+ \otimes S) \cap B(\overline{V})}\label{eq:43}\\
\end{array}
\end{equation}
\begin{equation}
\begin{array}{ll}
\label{eq:44}
& C \otimes \ZZ_3[x_{36},x_{48},x_{54}]\bigskip
        = \Frac{S \cap Z(\overline{V})}{S \cap B(\overline{V})}\\
\end{array}
\end{equation}

In particular, it follows from (\ref{eq:43}) 
that %
$D \otimes \ZZ_3[x_{36},x_{48},x_{54}]$ is an ideal of %
$\Cotor_{H^*(E_6)}^{}(\ZZ_3, \ZZ_3)$.

The remaining part is to check the splitness.
By the definition, it means that a relation in %
$C \otimes \ZZ_3[x_{36},x_{48},x_{54}]$ is also that of %
$\Cotor_{H^*(E_6)}^{} (\ZZ_3, \ZZ_3)$. This fact is a direct consequence 
of checking the relations stated in Theorem 2.3.B.

\section{The May spectral sequence}

In this section, we review our calculations 
from the view point of the May spectral sequence \cite[2]{May1}.
For the sake of convenience, we recall
the construction so that it fits to our situation.
Let us recall from \cite{KM} the following

\begin{Theorem}
{\it As an algebra, we have}
$$H^*(E_6) \cong \ZZ_3[x_8]/(x_8^3) 
                \otimes \Lambda(x_3,x_7,x_9,x_{11},x_{15},x_{17}),$$
{\it where the reduced diagonal map induced 
from the multiplication  $\wtilde{\varphi} 
        : \wtilde{H}^*(E_6) 
        \to \wtilde{H}^*(E_6) \otimes \wtilde{H}^*(E_6)$ is given by}
$$\begin{array}[t]{ll}
\wtilde{\varphi}(x_j) = x_8 \otimes x_{j-8} 
        & \quad for \quad j=11,15,17,\\
\wtilde{\varphi}(x_j) = 0 & \quad otherwise. 
\end{array}$$
\end{Theorem}

Let $A = H^*(E_6)$ and $I$ the augmented ideal generated by the 
elements $x_3$, $x_7$, $x_9$, $x_{11}$, $x_{15}$, $x_{17}$. %
Let $C_A^*(\ZZ_3, \ZZ_3)$ be the cobar complex \cite[2]{May1} and %
$[x_1 | \cdots | x_n]$ denote an element 
of $C_A^i(\ZZ_3, \ZZ_3)=I^{\otimes i}$ %
and $C_A^0(\ZZ_3, \ZZ_3) = [ \ ] \cong \ZZ_3$.
Then we can define a weight $v$ of $[x_1 | \cdots | x_n]$ by
$$\min \{n_1 + \cdots + n_i \mid x_j \in I^{n_j}, \ j =1, \cdots, i\} \,;$$
let $F^pC_A^*(\ZZ_3, \ZZ_3)$ be the ideal generated by the elements %
$[x_1 | \cdots | x_n]$ where $1 \le i$ such that %
$v([x_1 | \cdots | x_n]) \ge p$.

The construction of an injective resolution (see \cite{SI} and \cite{MS}) induces a surjection $p$ from $C_A^*(\ZZ_3, \ZZ_3)$ to $\overline{V}$ %
defined in Section 3;
more concretely it is described as $p([x_3]) = a_4$, \linebreak
$p([x_7]) = a_8$, $p([x_8]) = a_9$, $p([x_8^2]) = c_{17}$, 
$p([x_{11}]) = b_{12}$, $p([x_{15}]) = b_{16}$, $p([x_{17}]) = b_{18}$. 
We see that $w([x_8^2]) = 2$ and $w([x_i]) = 1$ for the other elements.
Since $w(a_9b_j) = w(b_ja_9 - c_{17}a_{j-8}) = 2$ for $j = 12,16,18$, 
we can introduce the filtration $F^p\overline{V}$ %
in $\overline{V}$ by $p(F^pC_A^*(\ZZ_3, \ZZ_3))$. 
We also use the same letter $w$ for the weight 
of an element of $\overline{V}$.
Then we have $w(c_{17}) = 2$ and $w(a_i) = w(b_i) = 1$.
We denote by $\{E_p(\overline{V}),d_p\}$ the spectral sequence 
of $\overline{V}$ induced from the above filtration.

\begin{Lemma}\label{52}
$($\rm{i}$)$ \quad {\it We have}
$$ E_1(\overline{V}) = \ZZ_3[x_{26}] 
        \otimes \ZZ_3[a_4,a_8,a_{10}] \otimes \Lambda(a_9)
        \otimes \ZZ_3[b_{12},b_{16},b_{18}],$$
{\it where $w(a_4) = w(a_8) = w(a_{10}) = w(a_9) 
        = w(b_{12}) = w(b_{16}) = w(b_{18}) = 1$ and $w(x_{26}) = 3$.}

$(\rm{ii})$ \quad {\it After the $E_p(\overline{V})$-term for $p \ge 1$, 
the spectral sequence coincides with the May spectral sequence
which converges to $\Cotor_{H^*(E_6)}^{} (\ZZ_3, \ZZ_3)$. }
\end{Lemma}
\begin{proof}
Since the only non trivial differential 
is given by $d_0(c_{17}) = a_9^2$, 
the argument to obtain $E_1(\overline{V})$ is the same as Section 3.
The rest of the lemma follows from the definitions of the May 
spectral sequence and the filtration.
\end{proof}

The next lemma follows from the proof of Lemma \ref{36}.

\begin{Lemma}\label{53}
{\it $E_2(\overline{V})$ is obtained 
from the differential $$d_1(b_{12}) = - a_4a_9, \quad
d_1(b_{16}) = - a_8a_9, \quad d_1(b_{18}) = - a_{10}a_9.$$
Hence $E_2(\overline{V})$ is given in the formula of Lemma \ref{36}
} 
\end{Lemma}

At this point, the meaning of Lemma 3.7 given by \cite{MS} becomes very clear.

\begin{Proposition}\label{54}
{\it $E_3(\overline{V})$ is obtained 
from the differential
$$d_2(a_9 Q) = - x_{26} \partial^2(Q) \quad \mbox{for} \quad 
        Q \in \ZZ_3[a_4,a_8,a_{10},b_{12},b_{16},b_{18}].$$
Furthermore, the May spectral sequence collapses at $E_3(\overline{V})$.}
\end{Proposition}

\begin{proof}
By Lemma \ref{37}, 
we have
$$x_{26}\partial^2(-Q) = d(a_9 Q + c_{17}\partial Q).$$
For the elements indicated from (3) to (7) at Lemma \ref{36}, 
it is easily shown that

\medskip\noindent
$$\begin{array}{l}
w(a_9 Q + c_{17}\partial Q) 
        = \min \{w(a_9 Q), w(c_{17}\partial Q) \} = 1 + w(Q), \\
w(x_{26} \partial^2(-Q)) 
        = w(x_{26}) + w(\partial^2(-Q)) = 3 + w(Q). \\
\end{array}$$

\medskip\noindent
So we obtain the assertion since $a_9Q + c_{17}\partial Q$ is 
represented by $a_9 Q$ at $E_2(V)$.
The other statement follows from the calculation of $E_7$ in Section 3.
\end{proof}

\bigskip
{\bf Remark 5.5} \quad The argument in this section is hinted by the second work 
of Evens and Siegel \cite{ES}. They have given an example of an ungraded
Hopf algebra with the type $P$ such that the May spectral sequence
does not collapse at the $E_2$-term. 
Corresponding to their example is ours in a graded sense.
The calculation of \cite{MS} for other exceptional Lie groups 
for odd prime seems probably to imply that the May spectral sequence
converging to $\Cotor_{H^*(G)}^{} (\ZZ_p, \ZZ_p)$ collapses 
at the $E_3$-level but does not at the $E_2$-level possibly
except the case ($E_8,3$).

%

\end{document}